\documentclass[reqno]{amsart}
\usepackage[utf8]{inputenc}
\usepackage{amscd,amsfonts,amsmath,amssymb,amsthm,bbm,dsfont,centernot,mathtools}
\usepackage{color,mathrsfs,stackrel}
\usepackage[margin=1in]{geometry}
\usepackage[colorlinks=true,linkcolor=blue,citecolor=red]{hyperref}
\usepackage{yhmath}
\usepackage{fancyhdr}
\usepackage{xy}
\usepackage{eso-pic}
\usepackage{xcolor}
\usepackage{extarrows}
\usepackage{relsize}
\usepackage{caption}
\usepackage{eurosym}
\usepackage{chngcntr}
\usepackage{tikz}

\newtheorem{corollary}{Corollary}[section]
\newtheorem{lemma}[corollary]{Lemma}
\newtheorem{proposition}[corollary]{Proposition}
\newtheorem{theorem}[corollary]{Theorem}

\theoremstyle{definition}
\newtheorem{definition}[corollary]{Definition}
\newtheorem{remark}[corollary]{Remark}

\newtheorem*{acknowledgements}{\sc Acknowledgements}

\numberwithin{equation}{section}

\allowdisplaybreaks

\newcommand{\xstararrow}[2]{\xrightharpoonup[#1]{#2\ \hspace{-0.02cm}_{*}\hspace{-0.2cm}}}

\def\Xint#1{\mathchoice
{\XXint\displaystyle\textstyle{#1}}%
{\XXint\textstyle\scriptstyle{#1}}%
{\XXint\scriptstyle\scriptscriptstyle{#1}}%
{\XXint\scriptscriptstyle\scriptscriptstyle{#1}}%
\!\int}
\def\XXint#1#2#3{{\setbox0=\hbox{$#1{#2#3}{\int}$ }
\vcenter{\hbox{$#2#3$ }}\kern-.6\wd0}}
\def\dashint{\Xint-}
\def\mint{\Xint{\rotatebox[origin][30]{$-$}}}

\def \div {\mathop {\rm div}\nolimits}
\def \Tr {\mathop {\rm Tr}\nolimits}
\def \de {\mathrm d}
\def \e {\mathrm e}
\def \R {\mathbb R}
\def \N {\mathbb N}
\def \C {\mathbb C}
\def \B {\mathbb B}
\def \V {\mathcal V}
\def \D {\mathcal D}

\begin{document}


\title[The viscoelastic paradox in a nonlinear Kelvin-Voigt type model of dynamic fracture]{The viscoelastic paradox in a nonlinear Kelvin-Voigt type model of dynamic fracture}

\author[M. Caponi]{Maicol Caponi}
\address[Maicol Caponi]{Dipartimento di Matematica e Applicazioni “R. Caccioppoli”, Università degli Studi di Napoli ``Federico II'', Via Cintia, Monte S. Angelo, 80126 Naples, Italy}
\email{maicol.caponi@unina.it}
\author[A. Carbotti]{Alessandro Carbotti}
\address[Alessandro Carbotti]{Dipartimento di Matematica e Fisica “E. De Giorgi”, Università del Salento, Via Per
Arnesano, 73100 Lecce, Italy}
\email{alessandro.carbotti@unisalento.it}
\author[F. Sapio]{Francesco Sapio}
\address[Francesco Sapio]{Wolfgang Pauli Institute, Oskar-Morgenstern-Platz 1, 1090 Vienna, Austria}
\email{sapiof23@univie.ac.at}

\begin{abstract}
In this paper we consider a dynamic model of fracture for viscoelastic materials, in which the constitutive relation, involving the Cauchy stress and the strain tensors, is given in an implicit nonlinear form. We prove the existence of a solution to the associated viscoelastic dynamic system on a prescribed time-dependent cracked domain via a discretisation-in-time argument. Moreover, we show that such a solution satisfies an energy-dissipation balance in which the energy used to increase the crack does not appear. As a consequence, in analogy to the linear case this nonlinear model exhibits the so-called \textit{viscoelastic paradox}.
\end{abstract}

\maketitle

\noindent
{\bf Keywords}: 
dynamic fracture, cracking domains, elastodynamics, nonlinear viscoelasticity, 
monotone operators, energy-dissipation balance, viscoelastic paradox

\medskip

\noindent
{\bf MSC 2020}: 35L53, 35A01, 35Q74, 47H05, 74D10, 
74R10.

	
\section{Introduction}

In the derivation of a mathematical model for dynamic crack propagation, the two fundamental facts that must be taken into account are the laws of elastodynamics and the (dynamic) Griffith criterion. The first one states that the displacement of the deformation must solve the elastodynamics system away from the crack, while the second one dictates how the crack grows in time. More precisely, the Griffith criterion (see~\cite{Grif,Mott}), originally formulated in the quasi-static setting, explains that there is a balance between the mechanical energy dissipated during the evolution and the energy used to increase the crack, which is supposed to be proportional to the area increment of the crack itself. 

The first step to address the study of a model of dynamic fracture is to find the solution to the elastodynamics system when the evolution of the crack is prescribed. From the mathematical point of view, this leads to the study of the following system in a time-dependent domain:
\begin{equation}\label{eq:elastodynamics}
\ddot u(t)-\div (\sigma(t))=f(t)\quad\text{in $\Omega\setminus\Gamma_t$, $t\in[0,T]$},
\end{equation}
with some prescribed boundary and initial conditions. In the above formulation $\Omega\subset\R^d$ is an open bounded set with Lipschitz boundary which represents the reference configuration of the material. The $(d-1)$-dimensional closed set $\Gamma_t\subset\overline\Omega$ models the crack at time $t$, $u(t)\colon\Omega\setminus\Gamma_t\to\R^d$ is the displacement of the deformation, $\sigma(t)$ is the Cauchy stress tensor, and $f(t)$ is a forcing term. Once we found the displacement $u$ which solves~\eqref{eq:elastodynamics} with a prescribed crack evolution $t\mapsto\Gamma_t$, we determine the pairs displacement-crack which satisfy the Griffith energy-dissipation balance. Finally, we select the “right" crack evolution according to some maximal dissipation principle. 

In the easiest case of a pure elastic material, the system 
~\eqref{eq:elastodynamics} is coupled with the following constitutive law involving the Cauchy stress and the strain tensors:
\begin{equation}\label{eq:el_conlaw}
\sigma(t)=\C eu(t)\quad\text{in $\Omega\setminus\Gamma_t$, $t\in[0,T]$},
\end{equation}
where $\C$ is the elasticity tensor, which is fourth order positive definite on the space of symmetric matrices $\R^{d\times d}_{\rm sym}$, and $eu=\frac{1}{2}(\nabla u+\nabla u^T)$ is the strain tensor. In this setting the Griffith criterion reads as
\begin{equation}\label{eq:el_enbal}
\frac{1}{2}\|\dot u(t)\|_2^2+\frac{1}{2}\|eu(t)\|_2^2+\mathcal H^{d-1}(\Gamma_t\setminus\Gamma_0)=\frac{1}{2}\|\dot u(0)\|_2^2+\frac{1}{2}\|eu(0)\|_2^2+\text{work of external forces}
\end{equation}
for all $t\in[0,T]$. We point out that the first two terms in the left-hand side of the above identity correspond to the mechanical energy (the sum of kinetic and elastic energy), while the term $\mathcal H^{d-1}(\Gamma_t\setminus\Gamma_0)$ models the energy used to increase the crack from $\Gamma_0$ to $\Gamma_t$. 

In the literature, we can find several mathematical results for the model associated to~\eqref{eq:elastodynamics} and~\eqref{eq:el_conlaw}. As for the existence of a solution when the evolution $t\mapsto\Gamma_t$ is prescribed, we refer to~\cite{DM-Lar,DM-Luc} for the antiplane case, that is when $u(t)\colon\Omega\setminus\Gamma_t\to\R$ is a scalar function and $eu$ is replaced by $\nabla u$, and~\cite{Caponi,Tasso} for the general case. Regarding the determination of the crack evolution $t\mapsto\Gamma_t$, we have only partial results. For example, we cite~\cite{CLT}, where the authors characterize in the antiplane case and for $d=2$ the pairs displacement-crack which satisfy the energy-dissipation balance, and~\cite{DMLT,DMLT2} in which for $d=2$ the authors studied the coupled problem under a suitable notion of maximal dissipation.

Viscoelastic materials, which exhibit both viscous and elastic behaviours when undergoing deformations, are another class widely studied in the literature. One of the simplest mathematical model is the Kelvin-Voigt one, where the constitutive law between the Cauchy stress and the strain tensors reads as
\begin{equation}\label{eq:KV_conlaw}
\sigma(t)=\C eu(t)+\mathbb B e\dot u(t)\quad\text{in $\Omega\setminus\Gamma_t$, $t\in[0,T]$},
\end{equation}
where $\C$ and $\B$ are the elasticity and the viscosity tensors, respectively. For the Kelvin-Voigt model, the Griffith criterion leads to the following energy-dissipation balance
\begin{equation}\label{eq:KV_enbal}
\begin{aligned}
&\frac{1}{2}\|\dot u(t)\|_2^2+\frac{1}{2}\|eu(t)\|_2^2+\mathcal H^{d-1}(\Gamma_t\setminus\Gamma_0)+\int_0^t\|e\dot u(s)\|_2^2\,\de s\\
&=\frac{1}{2}\|\dot u(0)\|_2^2+\frac{1}{2}\|eu(0)\|_2^2+\text{work of external forces}
\end{aligned}
\end{equation}
for all $t\in[0,T]$. Notice that, with respect to formula~\eqref{eq:el_enbal}, in~\eqref{eq:KV_enbal} we need to take into account also the energy dissipated by the viscous term, which is given by $\int_0^t\|e\dot u(s)\|_2^2\,\de s$. 

In~\cite{DM-Lar,Tasso} we can found existing results for the linear viscoelatic problem~\eqref{eq:elastodynamics} and~\eqref{eq:KV_conlaw}, when the evolution of the crack is prescribed. Unfortunately, in those papers, it is also shown that the Griffith energy-dissipation balance~\eqref{eq:KV_enbal} holds without the term $\mathcal H^{d-1}(\Gamma_t\setminus\Gamma_0)$. As a consequence, there is no pair displacement-crack which satisfies~\eqref{eq:KV_enbal}, unless the crack does not grow in time, i.e., $\Gamma_t=\Gamma_0$ for all $t\in[0,T]$. This phenomenon, which says that in the linear Kelvin-Voigt model the crack can not propagate, is well-known in mechanics as the {\it viscoelastic paradox}, see for instance~\cite[Chapter 7]{Slepyan}. We point out that, if the viscosity tensor $\B$ is allowed to degenerate in a neighborhood of the moving crack, the viscoelastic paradox does not occur, as shown in~\cite{CS}. For other versions of linear constitutive laws in the framework of viscoelatic materials, we refer for example to~\cite{CS2,Cianci,CDM,Sapio}.

More recently, viscoelastic materials in which the constitutive relation is nonlinear and given in an implicit form have been also considered. For example, in~\cite{BuPaSuSe}, the authors studied the following elastodynamic system in a domain without cracks: 
\begin{equation}\label{eq:elastodynamic2}
\ddot u(t)-\div (\sigma(t))=f(t)\quad\text{in $\Omega$, $t\in[0,T]$},
\end{equation}
with the implicit constitutive law
\begin{equation}
\label{eq:nonlinKV_conlaw2}G(\sigma(t))=eu(t)+e\dot u(t)\quad\text{in $\Omega$, $t\in[0,T]$},
\end{equation}
where $G\colon\R^{d\times d}_{sym}\to \R^{d\times d}_{sym}$ is a nonlinear monotone operator which satisfies suitable $p$-growth assumptions. In particular, the prototypical models studied are
\begin{equation}\label{eq:Gmodel}
G_1(\xi):=|\xi|^{p-2}\xi\quad\text{for $p>1$},\qquad G_2(\xi)=\frac{\xi}{(1+|\xi|^a)^\frac{1}{a}}\quad\text{for $p=1$, with $a>0$}. 
\end{equation}
As explained by Bulíček, Patel, Süli, and Şengül in their paper~\cite{BuPaSuSe} (see also~\cite{Patel2}), linear models may be inaccurate to describe real phenomena, while implicit constitutive theories allow for a more general structure in modelling than explicit ones. Moreover, as shown by Rajagopal in~\cite{Raja}, the nonlinear relationship between the stress and the strain can be obtained after linearising the strain, and so it make sense to consider implicit constitutive relations in the contest of small deformations. Under suitable assumptions on the initial data and on the nonlinear term $G$, Bulíček, Patel, Süli, and Şengül in~\cite{BuPaSuSe} proved existence and uniqueness of solutions to the problem~\eqref{eq:elastodynamic2} and~\eqref{eq:nonlinKV_conlaw2} via the Galerkin approximation.

The aim of our paper is to study the model of viscoelastic materials with implicit constitutive law of~\cite{BuPaSuSe}, in the framework of dynamic crack propagation. More precisely, we consider the elastodynamics system~\eqref{eq:elastodynamics} with the constitutive relation
\begin{equation}\label{eq:nonlinKV_conlaw}
G(\sigma(t))=eu(t)+e\dot u(t)\quad\text{in $\Omega\setminus\Gamma_t$, $t\in[0,T]$},
\end{equation}
where $G\colon\R^{d\times d}_{sym}\to \R^{d\times d}_{sym}$ is a nonlinear monotone operator which satisfies suitable $p$-growth assumptions (more precisely (G1)--(G3) in Section~\ref{sec:not_prob}). Since the linear growth $p=1$ is hard to handle even in the case with no cracks, we restrict ourselves to the range $p\in (1,2^*)$, where $2^*:=\frac{2d}{d-2}$ is the Sobolev critical exponent. The condition $p<2^*$, which also appears in ~\cite{BuPaSuSe}, is needed to ensure that the displacement $u(t)$ is an element of $L^2(\Omega\setminus\Gamma_t;\R^d)$. Indeed, from~\eqref{eq:nonlinKV_conlaw}, we easily deduce that $u(t)$ lives in the Sobolev space $W^{1,p'}(\Omega\setminus\Gamma_t;\R^d)$, being $p'=\frac{p}{p-1}$ the Hölder conjugate exponent of $p$, which is compactly embedded in $L^2(\Omega\setminus\Gamma_t;\R^d)$ whenever $p<2^*$. This simplifies the mathematical formulation of the problem. An interesting question, which is out of the scope of this paper, is whether this condition can be removed.

Our first result is Theorem~\ref{thm:main}, where we prove the existence of a solution to the problem~\eqref{eq:elastodynamics} and~\eqref{eq:nonlinKV_conlaw} when the crack evolution $t\mapsto\Gamma_t$ is prescribed, under suitable conditions on the data and on the nonlinear term $G$. The proof of Theorem~\ref{thm:main} follows the main ideas of~\cite{BuPaSuSe}, adapting them to our setting. First, since the Galerkin approximation does not fit well with the framework of time-dependent domains, we use the discretisation-in-time scheme exploited in~\cite{DM-Lar}. Moreover, since we want to consider nonlinear operators which are not strictly monotone, we regularise $G$ in order to invert the relation~\eqref{eq:nonlinKV_conlaw}. This allows us to write the Cauchy tensor in terms of the displacement and to switch from the formulation~\eqref{eq:elastodynamics} and~\eqref{eq:nonlinKV_conlaw} to a simpler system. More precisely, we fix $n\in\N$ and we search a discrete-in-time approximate solution to~\eqref{eq:elastodynamics} and~\eqref{eq:nonlinKV_conlaw} with $G$ replaced by its regularisation. Then, we perform a discrete energy estimate (see Lemma~\ref{lem:estM}), which allows us to pass to the limit as $n\to\infty$ to obtain a pair $(u,\sigma)$ which solves~\eqref{eq:elastodynamics}. We prove that the displacement $u$ is more regular in time, and by using a standard technique in the monotone operator theory, we show that $(u,\sigma)$ satisfies also the implicit constitutive relation~\eqref{eq:nonlinKV_conlaw}. We conclude this part of the paper with Theorem~\ref{thm:uniq}, where we prove that there is at most one pair $(u,\sigma)$ with the same regularity of the solution of Theorem~\ref{thm:main} and that solves~\eqref{eq:elastodynamics} and~\eqref{eq:nonlinKV_conlaw} for a prescribed crack evolution $t\mapsto\Gamma_t$.

In the second part of the paper, we aim to study the validity of the Griffith energy-dissipation balance for the implicit nonlinear model~\eqref{eq:elastodynamics} and~\eqref{eq:nonlinKV_conlaw}. At first, in Theorem~\ref{thm:enbal} we prove that the mechanical energy of every regular solution to problem~\eqref{eq:elastodynamics} and~\eqref{eq:nonlinKV_conlaw} (in particular, of the one found in Theorem~\ref{thm:main}) satisfies the implicit energy balance~\eqref{eq:enbal}. Then, we consider the strictly monotone operator $G(\xi)=|\xi|^{p-2}\xi$, so that our problem reduces to the nonlinear Kelvin-Voigt system
\begin{equation}\label{eq:elasto_nonlineKV}
\ddot u(t)-\div(|eu(t)+e\dot u(t)|^{p'-2}(eu(t)+e\dot u(t)))=f(t)\quad\text{in $\Omega\setminus\Gamma_t$, $t\in[0,T]$}. 
\end{equation}
 In this setting, the Griffith energy-dissipation balance takes the form
\begin{equation}\label{eq:nonlin_enbal}
\begin{aligned}
&\frac{1}{2}\|\dot u(t)\|_2^2+\frac{1}{p'}\|eu(t)\|_{p'}^{p'}+\mathcal H^{d-1}(\Gamma_t\setminus\Gamma_0)\\
&+\int_0^t\int_\Omega\left(|eu(s,x)+e\dot u(s,x)|^{p'-2}(eu(s,x)+e\dot u(s,x))-|eu(s,x)|^{p'-2}eu(s,x)\right)\cdot e\dot u(s,x)\,\de x\,\de s\\
&=\frac{1}{2}\|\dot u(0)\|_2^2+\frac{1}{p'}\|eu(0)\|_{p'}^{p'}+\text{work of external forces}
\end{aligned}
\end{equation}
for every $t\in[0,T]$. In particular, the energy dissipated by the viscous term is given by $$\int_0^t\int_\Omega\left(|eu(s,x)+e\dot u(s,x)|^{p'-2}(eu(s,x)+e\dot u(s,x))-|eu(s,x)|^{p'-2}eu(s,x)\right)\cdot e\dot u(s,x)\,\de x\,\de s\ge 0,$$
which reduces to the corresponding term in~\eqref{eq:KV_enbal} for $p=2$ (notice that this term is non negative due to the monotonicity of $G^{-1}(\eta):=|\eta|^{p'-2}\eta$). For this particular choice of $G$, in Corollary~\ref{coro:enbal2} we derive that the energy dissipation balance proved in Theorem~\ref{thm:enbal} can be rewritten just in terms of the displacement $u$ as~\eqref{eq:enbal2}. Therefore, the pair displacement-crack given by Theorem~\ref{thm:main} satisfies~\eqref{eq:nonlin_enbal} if and only if $\Gamma_t=\Gamma_0$ for every $t\in[0,T]$, i.e., when the crack does not grow in time. This shows that also the nonlinear Kelvin-Voigt model of dynamic fracture exhibits the viscoelastic paradox, as it happens in~\cite{DM-Lar,Tasso} for the corresponding linear model.

We conclude the introduction by observing that the corresponding phase-field model of dynamic crack propagation has been analyzed by~\cite{Patel} (see also~\cite{Patel2}). This is the one in which, roughly speaking, for a fixed $\epsilon>0$ the crack set is replaced by a function $v_\epsilon$ which is 0 in a $\epsilon$-neighborhood of the crack and 1 far from it. More precisely, in~\cite{Patel} the author proved that there exists a pair $(u_\epsilon,v_\epsilon)$ which satisfies the elastodynamics system with the implicit constitutive law and the Griffith energy-dissipation balance for both the nonlinearities in~\eqref{eq:Gmodel}. Therefore, it could be interesting to understand in a future paper if there is a connection between these two models; in particular, if the viscoelastic paradox can also occur in the phase-field setting. 

The rest of the paper goes as follows: in Section~\ref{sec:not_prob} we introduce the mathematical framework of our model of dynamic fracture for viscoelastic material, and we fix the main assumptions on the reference set, the crack evolution, and the nonlinearity $G$ in the constitutive law. Moreover, in Definition~\ref{def:weak_sol} we give the notion of (weak) solution to problem~\eqref{eq:elastodynamics} and~\eqref{eq:nonlinKV_conlaw}, and we state our main existence result, which is Theorem~\ref{thm:main}. In Section~\ref{sec:exis} we prove Theorem~\ref{thm:main} by performing a discretisation-in-time scheme together with a regularisation of the nonlinearity $G$. At first, we find an approximate solution in each node of the discretisation of the regularised model. Then, in Lemma~\ref{lem:estM} we prove a discrete energy estimate, which allows us to pass to the limit when the parameter of the discretisation and regularisation goes to $0$. Finally, we show that under suitable regularity assumptions the solution is unique. We conclude the paper with Section~\ref{sec:enbalandvp}, where we prove that every regular solution to~\eqref{eq:elastodynamics} and~\eqref{eq:nonlinKV_conlaw} satisfies the energy-dissipation identity of Theorem~\ref{thm:enbal}. Afterwards, we consider the nonlinear Kelvin-Voigt system~\eqref{eq:elasto_nonlineKV}, and we use the energy-dissipation identity to show that this model exhibits the viscoelastic paradox.


\section{Notation and formulation of the model}\label{sec:not_prob}

\subsection{Notation} The space of $m\times d$ matrices with real entries is denoted by $\R^{m\times d}$; in case $m=d$, the subspace of symmetric matrices is denoted by $\R^{d\times d}_{sym}$. For any $A,B\in\R^{d\times d }$ we denote with $A\cdot B$ the Frobenius scalar product, namely $A\cdot B:=\Tr(A^TB)$. Given a function $u\colon\R^d\to\R^m$, we denote its Jacobian matrix by $\nabla u$, whose components are $(\nabla u)_{ij}:= \partial_j u_i$ for $i\in\{1,\dots,m\}$ and $j\in\{1,\dots,d\}$; when $u\colon \R^d\to\R^d$, we use $eu$ to denote the symmetric part of the gradient, namely $eu:=\frac{1}{2}(\nabla u+\nabla u^T)$. Given a tensor field $A\colon \R^d\to\R^{m\times d}$, by $\div A$ we mean its divergence with respect to rows, namely $(\div A)_i:= \sum_{j=1}^d\partial_jA_{ij}$ for $i\in\{1,\dots,m\}$. 

We denote the $d$-dimensional Lebesgue measure by $\mathcal L^d$ and the $(d-1)$-dimensional Hausdorff measure by $\mathcal H^{d-1}$; given a bounded open set $\Omega$ with Lipschitz boundary, by $\nu$ we mean the outer unit normal vector to $\partial\Omega$, which is defined $\mathcal H^{d-1}$-a.e.\ on the boundary. The Lebesgue and Sobolev spaces on $\Omega$ are defined as usual; the boundary values of a Sobolev function are always intended in the sense of traces. 

The norm of a generic Banach space $X$ is denoted by $\|\cdot\|_X$; when $X$ is a Hilbert space, we use $(\cdot,\cdot)_X$ to denote its scalar product. We denote by $X'$ the dual of $X$ and by $\langle \cdot,\cdot\rangle_{X'}$ the duality product between $X'$ and $X$. Given two Banach spaces $X_1$ and $X_2$, the space of linear and continuous maps from $X_1$ to $X_2$ is denoted by $\mathscr L(X_1;X_2)$; given $\mathbb A\in\mathscr L(X_1;X_2)$ and $u\in X_1$, we write $\mathbb A u\in X_2$ to denote the image of $u$ under $\mathbb A$. 

Given an open interval $(a,b)\subset\R$ and $q\in[1,\infty]$, we denote by $L^q(a,b;X)$ the space of $L^q$ functions from $(a,b)$ to $X$; we use $W^{k,q}(a,b;X)$ to denote the Sobolev space of functions from $(a,b)$ to $X$ with derivatives up to order $k$ in $L^q(a,b;X)$. Given $u\in W^{1,q}(a,b;X)$, we denote by $\dot u\in L^q(a,b;X)$ its derivative in the sense of distributions. When dealing with an element $u\in W^{1,q}(a,b;X)$ we always assume $u$ to be the continuous representative of its class, and therefore, the pointwise value $u(t)$ of $u$ is well defined for all $t\in[a,b]$. We use $C_w^0([a,b];X)$ to denote the set of weakly continuous functions from $[a,b]$ to $X$, namely, the collection of maps $u\colon [a,b]\to X$ such that $t\mapsto \langle x',u(t)\rangle_{X'}$ is continuous from $[a,b]$ to $\R$, for all $x'\in X'$. 

\subsection{Mathematical framework} Let $T>0$ and $d\in\N$ with $d\ge 2$. Let $\Omega\subset\R^d$ be a bounded open set (which represents the reference configuration of the body) with Lipschitz boundary. Let $\partial_D\Omega$ be a Borel subset of $\partial\Omega$, on which we prescribe the Dirichlet condition, $\partial_N\Omega$ its complement in $\partial\Omega$, and $\Gamma\subset\overline\Omega$ the prescribed crack path. As in~\cite{CS,CS2}, we assume the following hypotheses on the geometry of the crack and the Dirichlet part of the boundary:
\begin{itemize}
\item[(E1)] $\Gamma$ is a closed set with $\mathcal L^d(\Gamma)=0$ and $\mathcal H^{d-1}(\Gamma\cap\partial\Omega)=0$;
\item[(E2)] $\Omega\setminus\Gamma$ is the union of two disjoint bounded open sets $\Omega_1$ and $\Omega_2$ with Lipschitz boundary;
\item[(E3)] $\partial_D\Omega\cap\partial\Omega_i$ contains the graph of a Lipschitz function $\theta_i$ over a non empty open subset of $\R^{d-1}$ for all $i\in\{1,2\}$;
\item[(E4)] $\{\Gamma_t\}_{t\in[0,T]}$ is a family of closed subsets of $\Gamma$ satisfying $\Gamma_s\subset\Gamma_t$ for all $0< s\le t\le T$.
\end{itemize}
We recall that the set $\Gamma_t$ represents the prescribed crack at time $t\in[0,T]$ inside $\Omega$.

Thanks to (E1)--(E4) for all $q\in[1,\infty]$ the space $L^q(\Omega\setminus\Gamma_t;\R^d)$ coincides with $L^q(\Omega;\R^d)$ for all $t\in[0,T]$. In particular, we can extend a function $u\in L^q(\Omega\setminus\Gamma_t;\R^d)$ to a function in $L^q(\Omega;\R^d)$ by setting $u=0$ on $\Gamma_t$. Moreover, for all $q\in[1,\infty)$ the trace of $u\in W^{1,q}(\Omega\setminus\Gamma;\R^d)$ is well defined on $\partial\Omega$ and there exists a constant $C_{tr}>0$, depending on $\Omega$, $\Gamma$, and $q$, such that
\begin{equation}\label{eq:traceM}
\|u\|_{L^q(\partial\Omega;\R^d)}\le C_{tr}\|u\|_{W^{1,q}(\Omega\setminus\Gamma;\R^d)}\quad\text{for all }u\in W^{1,q}(\Omega\setminus\Gamma;\R^d).
\end{equation}
Hence, we can define the space
$$W^{1,q}_D(\Omega\setminus\Gamma;\R^d):=\{u\in W^{1,q}(\Omega\setminus\Gamma;\R^d)\,:\, u=0\text{ on }\partial_D\Omega\}.$$
Furthermore, by using the second Korn inequality in $\Omega_1$ and $\Omega_2$ (see, e.g.,~\cite[Theorem~2.4]{OSY}) and taking the sum we can find a positive constant $C_K$, depending on $\Omega$, $\Gamma$, and $q$ such that 
\begin{equation}\label{eq:kornM}
\|\nabla u\|_{L^q(\Omega;\R^{d\times d})}\le C_K(\|u\|_{L^q(\Omega;\R^d)}^q+\|eu\|_{L^q(\Omega;\R^{d\times d}_{sym})}^q)^{\frac{1}{q}}\quad\text{for all }u\in W^{1,q}(\Omega\setminus\Gamma;\R^d).
\end{equation}
Similarly, thanks to the Korn-Poincaré inequality (see, e.g.,~\cite[Theorem~2.7]{OSY}) we obtain also the existence of a constant $C_{KP}$, depending on $\Omega$, $\Gamma$, $q$, and $\partial_D\Omega$, such that 
\begin{equation}\label{eq:kornpoi}
\|u\|_{W^{1,q}(\Omega\setminus\Gamma;\R^d)}\le C_{KP}\|eu\|_{L^q(\Omega;\R^{d\times d}_{sym})}\quad\text{for all $u\in W^{1,q}_D(\Omega\setminus\Gamma;\R^d)$}.
\end{equation}
Finally, for all $q\in(\frac{2d}{d+2},\infty]$ the embedding $W^{1,q}(\Omega\setminus\Gamma;\R^2)\hookrightarrow L^2(\Omega;\R^2)$ is continuous and compact. 

We fix $p\in (1,2^*)$, where $2^*$ is the Sobolev conjugate of $2$, defined as
$$2^*:=\begin{cases}
\infty &\text{for $d=2$},\\
\frac{2d}{d-2}&\text{for $d>2$}.
\end{cases}$$
Notice that $p\in (1,2^*)$ if and only if $p'\in (\frac{2d}{d+2},\infty)$, where $p':=\frac{p}{p-1}$ is the Hölder conjugate exponent of $p$. We set $H:=L^2(\Omega;\R^d)$ and we define the following spaces
\begin{equation*}
V:=W^{1,p'}(\Omega\setminus\Gamma;\R^d)\quad \text{ and }\quad V_t:= W^{1,p'}(\Omega\setminus\Gamma_t;\R^d)\quad \text{for all $t\in [0,T]$}.
\end{equation*}
We point out that in the definition of $V$ and $V_t$, we are considering only the distributional gradient of $u$ in $\Omega\setminus\Gamma$ and in $\Omega\setminus\Gamma_t$, respectively, and not the one in $\Omega$. Taking into account~\eqref{eq:kornM}, we shall use on the set $V_t$ (and also on the set $V$) the equivalent norm
\begin{equation*}
\|u\|_{V_t}:=\left(\|u\|_{p'}^{p'}+\|eu\|_{p'}^{p'}\right)^{\frac{1}{{p'}}}\quad\text{for all }u\in V_t.
\end{equation*}
Furthermore, by~\eqref{eq:traceM}, we can consider the sets
\begin{equation*}
V^D:=\{u\in V\,:\,u=0\text{ on }\partial_D\Omega\},\qquad V_t^D:=\{u\in V_t\,:\,u=0\text{ on }\partial_D\Omega\}\quad\text{for all $t\in[0,T]$},
\end{equation*}
which are closed subspaces of $V$ and $V_t$, respectively. 

\begin{remark}\label{rem:comp}
Since $p\in(1,2^*)$, by exploiting (E1)--(E4) we derive that for all $t\in[0,T]$ the space $V_t^D$ is a separable reflexive Banach space with embedding 
$$V_t^D\hookrightarrow H\text{ continuous, compact, and dense}.$$
In particular, the aforementioned condition on $p$ is used to deduce the compactness of $V_t^D$ in $H$. Therefore, the embedding $H\hookrightarrow (V_t^D)'$, which is defined by 
\begin{equation}\label{eq:dual}
\langle h,u\rangle_{(V_t^D)'}:=(h,u)_H\quad\text{for $h\in H$ and $u\in V_t^D$},
\end{equation}
is continuous, and the same holds true also for $V_t$, $V$, and $V^D$.
\end{remark}

Let us consider a nonlinear operator $G\colon \R^{d\times d}_{sym}\to\R^{d\times d}_{sym}$ satisfying the following assumptions:
\begin{itemize}
\item[(G1)] there exists a convex function $\phi\colon \R^{d\times d}_{sym}\to \R$ of class $C^1$ such that $G(\xi)=\nabla\phi(\xi)$ for all $\xi\in\R^{d\times d}_{sym}$;
\item[(G2)] there exist constants $b_1>0$ and $b_2\ge 0$ such that $G(\xi)\cdot \xi\ge b_1|\xi|^p-b_2$ for all $\xi\in\R^{d\times d}_{sym}$;
\item[(G3)] there exists a constant $b_3> 0$ such that $|G(\xi)|\le b_3(1+|\xi|^{p-1})$ for all $\xi\in\R^{d\times d}_{sym}$.
\end{itemize}

\begin{remark}\label{rem:phi0}
The assumption (G1) implies that $G$ is continuous and monotone, i.e.,
\begin{equation}\label{eq:monoton}
 (G(\xi_1)-G(\xi_2))\cdot (\xi_1-\xi_2)\ge 0\quad\text{for all $\xi_1,\xi_2\in\R^{d\times d}_{sym}$}.
\end{equation}
Moreover, up to add a constant, we always assume that $\phi(0)=0$.
\end{remark}

Given
\begin{itemize}
\item [(D1)] $f\in L^2(0,T;H)$;
\item [(D2)] $z\in W^{2,p'}(0,T;V_0)$;
\item [(D3)] $u^0,u^1\in V_0$ such that $u^0-z(0)\in V_0^D$ and $u^1-\dot z(0)\in V_0^D$;
\end{itemize}
we study the following dynamic viscoelastic system with implicit nonlinear constitutive law:
\begin{equation}\label{eq:nonlin-KV}
\begin{cases}
\ddot u(t)-\div(\sigma(t))=f(t)&\text{in $\Omega\setminus\Gamma_t$, $t\in[0,T]$},\\
G(\sigma(t))=eu(t)+e\dot u(t)&\text{in $\Omega\setminus\Gamma_t$, $t\in[0,T]$},
\end{cases}
\end{equation}
equipped with the boundary conditions
\begin{alignat}{4}
&u(t)=z(t) && \quad \text{on $\partial_D\Omega$}, & \quad t\in[0,T],\\
&\sigma(t)\nu=0 &&\quad\text{on $\partial_N\Omega\cup \Gamma_t$}, & \quad t\in[0,T],\label{eq:boundaryN}
\end{alignat}
where $\nu$ denotes the outward unit normal to $\partial\Omega$, and the initial conditions
\begin{alignat}{4}
&u(0)=u^0,\quad \dot{u}(0)=u^1&&\quad\text{in $\Omega\setminus \Gamma_0$}.\label{eq:initial}
\end{alignat}
Notice that in~\eqref{eq:nonlin-KV}--\eqref{eq:initial} the explicit dependence on $x$ is omitted to enlighten notation. As usual, the Neumann boundary conditions are only formal, and their meaning will be explained in Remark~\ref{rem:neumann}. 

From now on we always assume that $p\in(1,2^*)$ and that (E1)--(E4), (G1)--(G3), and (D1)--(D3) are satisfied. Let us define the following functional spaces:
\begin{align*}
&\mathcal V:=\{\varphi\in W^{1,p'}(0,T;V)\cap W^{1,\infty}(0,T;H)\,:\,\varphi(t)\in V_t\text{ for all $t\in[0,T]$}\},\\
&\D:=\{\varphi\in C^1_c(0,T;V)\,:\,\varphi(t)\in V_t^D\text{ for all $t\in[0,T]$}\}.
\end{align*}

Similarly to~\cite{DMT}, we introduce the following notion of weak solution.

\begin{definition}[Weak solution]\label{def:weak_sol}
A pair $(u,\sigma)\in\V\times L^p(0,T;L^p(\Omega,\R^{d\times d}_{sym}))$ is a {\it weak solution} to the nonlinear viscoelastic system~\eqref{eq:nonlin-KV}--\eqref{eq:boundaryN} if
\begin{itemize}
\item[(i)] $u(t)-z(t)\in V_t^D$ for all $t\in[0,T]$;
\item[(ii)] the following identity holds 
\begin{equation}\label{eq:wweak}
-\int_0^T(\dot u(t),\dot \varphi(t))_H\,\de t+\int_0^T(\sigma(t),e \varphi(t))_{p,p'}\,\de t=\int_0^T( f(t),\varphi (t))_H\,\de t\quad\text{or all $\varphi\in\D$,}
\end{equation}
where $(g,h)_{p,p'}:=\int_\Omega g(x)\cdot h(x)\,\de x$ for all $g\in L^p(\Omega;\R^{d\times d}_{sym})$ and $h\in L^{p'}(\Omega;\R^{d\times d}_{sym})$;
\item[(iii)] the constitutive law 
\begin{equation}\label{eq:conlaw}
G(\sigma(t))=eu(t)+e\dot u(t)\quad \text{in $\Omega\setminus\Gamma_t$ for a.e.\ $t\in[0,T]$}
\end{equation}
is satisfied.
\end{itemize}

\end{definition}

\begin{remark}\label{rem:neumann}
The Neumann boundary conditions~\eqref{eq:boundaryN} are formally used to pass from the strong formulation~\eqref{eq:nonlin-KV}--\eqref{eq:boundaryN} to the weak formulation~\eqref{eq:wweak}. Notice that, if $u(t)$, $\sigma(t)$, and $\Gamma_t$ are sufficiently regular, then~\eqref{eq:boundaryN} can be deduced from~\eqref{eq:wweak} by using integration by parts in space.
\end{remark}

We want to give a meaning to the initial conditions~\eqref{eq:initial} for a weak solution $(u,\sigma)$ to~\eqref{eq:nonlin-KV}--\eqref{eq:boundaryN}. To this aim, we first recall the following result (see, for instance~\cite[Chapitre XVIII, \S 5, Lemme 6]{DL}).

\begin{lemma}\label{lem:wcM}
Let $X,Y$ be reflexive Banach spaces such that $X\hookrightarrow Y$ continuously. Then 
$$L^{\infty}(0, T;X)\cap C^0_w([0, T];Y)= C^0_w([0, T];X).$$
\end{lemma}

Moreover, we need the following regularity result for the weak solutions to~\eqref{eq:nonlin-KV}--\eqref{eq:boundaryN}.

\begin{lemma}
Let $(u,\sigma)\in\mathcal V\times L^p(0,T;L^p(\Omega,\R^{d\times d}_{sym}))$ be a weak solution to the nonlinear viscoelastic system~\eqref{eq:nonlin-KV}--\eqref{eq:boundaryN}. Then $u\in W^{2,q}(0,T;(V_0^D)')$, where $q:=\min\{2,p\}$. In particular $u\in C^0([0,T];V)$ and $\dot u\in C^0_w([0,T];H)$.
\end{lemma}

\begin{proof}
Let us set $q:=\min\{2,p\}$. We define $\Lambda\in L^q(0,T;(V^D_0)')$ in the following way: 
\begin{equation*}
 \langle \Lambda(t),v\rangle_{(V_0^D)'}:=-(\sigma(t),ev)_{p,p'}+(f (t),v)_H\quad \text{for all $v\in V^D_0$ and for a.e.\ $t\in [0,T]$}.
\end{equation*}
Let us consider a test function $\psi\in C^{1}_c(0,T)$, then for all $v\in V^D_0$ the function $\varphi(t):=\psi(t)v$ satisfies
\begin{equation}\label{eq:app}
 \varphi\in C^{1}_c(0,T;V),\quad \varphi(t)\in V_0^D\subset V_t^D \quad\text{for all $t\in[0,T]$}.
\end{equation}
Thanks to~\eqref{eq:wweak}, since $\varphi\in \D$ from~\eqref{eq:app}, we can write
\begin{align*}
-\int_0^T(\dot u(t),v)_H\dot\psi(t)\,\de t&=-\int_0^T(\sigma(t),ev)_{p,p'}\psi(t)\,\de t+\int_0^T(f (t),v)_H\psi(t)\,\de t=\int_0^T\langle \Lambda(t),v\rangle_{(V_0^D)'}\psi(t)\,\de t,
\end{align*}
which implies by~\eqref{eq:dual}
\begin{equation*}
\Big\langle-\int_0^T\dot u(t)\dot\psi(t)\,\de t,v\Big\rangle_{(V_0^D)'}=\Big\langle \int_0^T \Lambda(t)\psi(t)\,\de t,v\Big\rangle_{(V_0^D)'}\quad \text{for all $v\in V^D_0$}.
\end{equation*}
Hence, we get
\begin{equation}\label{eq:dersecq}
-\int_0^T\dot u(t)\dot\psi(t)\,\de t= \int_0^T \Lambda(t)\psi(t)\,\de t\quad \text{in $(V^D_0)'$ for all $\psi\in C^{1}_c(0,T)$.}
\end{equation}
Since $\dot u\in L^\infty(0,T;H)\hookrightarrow L^\infty(0,T;(V^D_0)')$ then identity~\eqref{eq:dersecq} implies
\begin{equation*}
u\in W^{2,q}(0,T;(V^D_0)').
\end{equation*}

Therefore $\dot u\in W^{1,q}(0,T;(V^D_0)')\hookrightarrow C^0([0,T];(V^D_0)')$, and since $\dot u\in L^{\infty}(0,T;H)$ by Lemma~\ref{lem:wcM} we deduce that $\dot u\in C^0_w([0,T];H)$. Finally, we have $W^{1,p'}(0,T;V)\hookrightarrow C^0([0,T];V)$ hence $u\in C^0([0,T];V)$.
\end{proof}

If $(u,\sigma)\in\V\times L^p(0,T;L^p(\Omega;\R^{d\times d}_{sym}))$ is a weak solution to~\eqref{eq:nonlin-KV}--\eqref{eq:boundaryN}, then $u(t)$ and $\dot u(t)$ are well defined as functions of $V$ and $H$, respectively, for all $t\in[0,T]$. Hence, it makes sense to evaluate them at time $t=0$ and we can introduce the following definition.

\begin{definition}[Initial conditions]
We say that a weak solution $(u,\sigma)\in\V\times L^p(0,T;L^p(\Omega;\R^{d\times d}_{sym}))$ to the nonlinear viscoelastic system~\eqref{eq:nonlin-KV}--\eqref{eq:boundaryN} satisfies the initial conditions~\eqref{eq:initial} if
\begin{equation*}
 u(0)=u^0\quad\text{in $V$},\quad \dot u(t)=u^1\quad\text{in $H$}
\end{equation*}
\end{definition}

The main existence result of this paper is the following theorem.

\begin{theorem}\label{thm:main}
There exists a weak solution $(u,\sigma)\in\mathcal V\times L^p(0,T;L^p(\Omega;\R^{d\times d}_{sym}))$ to the nonlinear viscoelastic system~\eqref{eq:nonlin-KV}--\eqref{eq:boundaryN} satisfying the initial conditions~\eqref{eq:initial}. Moreover, $u\in W^{2,2}(0,T;H)$.
\end{theorem}

The proof of Theorem~\ref{thm:main} is postponed to the next section. We point out that the displacement $u$ of the solution found in Theorem~\ref{thm:main} is more regular in time, more precisely $\ddot u\in L^2(0,T;H)$. This regularity is used at the end of Section~\ref{sec:exis} to prove a uniqueness result for the nonlinear viscoelastic system ~\eqref{eq:nonlin-KV}--\eqref{eq:initial}. Moreover, we exploit such a regularity in Section~\ref{sec:enbalandvp} to show the energy-dissipation balance of Theorem~\ref{thm:enbal}. This identity implies the viscoelastic paradox, which is discussed at the end of the paper.


\section{Existence of solutions}\label{sec:exis}

This section is devoted to the proof of Theorem~\ref{thm:main}. As explained in the introduction, the main idea is to combine the discretisation-in-time scheme of~\cite{DM-Lar} with the regularisation of the nonlinear operator $G$ introduced in~\cite{BuPaSuSe}. Therefore, we rephrase the system~\eqref{eq:nonlin-KV} in a simpler way, and we use Browder-Minty Theorem to find a sequence of approximate solutions in each node of the discretisation scheme. Then in Lemma~\ref{lem:estM} we prove a discrete energy estimate and we use a compactness argument to obtain a pair $(u,\sigma)$ which solves~\eqref{eq:wweak} (see Lemma~\ref{lem:wweak}). Finally, in Lemma~\ref{lem:conlaw}, by performing a standard argument in the theory of nonlinear monotone operators we show the validity of the constitutive law~\eqref{eq:conlaw}.

Let us fix $n\in\N$ and set
\begin{align*}
&\tau_n:=\frac{T}{n}, & & u_n^0:=u^0, & &V_n^k:=V^D_{k\tau_n},\qquad z_n^k:=z(k\tau_n),& &\text{for $k\in\{0,\dots,n\}$},\\
&\delta u_n^0:=u^1, & & \delta z_n^0:=\dot z(0), & & f_n^k:=\dashint_{(k-1)\tau_n}^{k\tau_n} f(t)\de t, & &\text{for $k\in\{1,\dots,n\}$}.
\end{align*}
We define $G_n\colon \R^{d\times d}_{sym}\to\R^{d\times d}_{sym}$ as
$$G_n(\xi):=G(\xi)+\frac{1}{n}|\xi|^{p-2}\xi\quad\text{for all $\xi\in\R^{d\times d}_{sym}$}.$$
Notice that $G_n$ still satisfies (G1)--(G3) with $\phi$ replaced by 
$$\phi_n(\xi):=\phi(\xi)+\frac{1}{np}|\xi|^p\quad\text{for all $\xi\in\R^{d\times d}_{sym}$},$$
and with $b_3$ replaced by $b_3+1$. Since $G_n$ is strictly monotone, by the standard theory of monotone operators there exists the inverse operator $G_n^{-1}\colon \R^{d\times d}_{sym}\to\R^{d\times d}_{sym}$, which is still strictly monotone. Moreover, if we introduce the Legendre transform $\phi_n^*$ of $\phi_n$, defined as
$$\phi_n^*(\eta):=\sup_{\xi\in\R^{d\times d}_{sym}}\{\eta\cdot \xi-\phi_n(\xi)\}\quad\text{for all $\eta\in\R^{d\times d}_{sym}$},$$
by (G1)--(G3) we have that $\phi_n^*\colon\R^{d\times d}_{sym}\to\R$ is still a convex function of class $C^1$ and $G_n^{-1}$ satisfies 
\begin{alignat}{3}
&G_n^{-1}(\eta)=\nabla\phi_n^*(\eta)&&\quad\text{for all $\eta\in\R^{d\times d}_{sym}$},\label{eq:G1n}\\
& G_n^{-1}(\eta)\cdot \eta\ge c_1|\eta|^{p'}- c_2 &&\quad\text{for all $\eta\in\R^{d\times d}_{sym}$},\label{eq:G2n}\\
&|G_n^{-1}(\eta)|\le c_3(1+|\eta|^{p'-1}) &&\quad \text{for all $\eta\in\R^{d\times d}_{sym}$},\label{eq:G3n}
\end{alignat}
for suitable constants $c_1,c_3>0$ and $c_2\ge 0$ independent of $n\in\N$. Furthermore, if we define $\eta_0:=G(0)=G_n(0)$, by the assumption $\phi(0)=0$ (see Remark~\ref{rem:phi0}) we have
$$\phi_n^*(\eta_0)=-\phi_n(0)=0.$$
Therefore, thanks to the convexity of $\phi_n^*$ we derive
\begin{align}
&\phi_n^*(\eta)\ge \phi_n^*(\eta_0)+G_n^{-1}(\eta_0)\cdot (\eta-\eta_0)=0\quad\text{for all $\eta\in\R^{d\times d}_{sym}$},\label{eq:phin0}\\
&\phi_n^*(\eta)\le\phi_n^*(\eta_0)+G_n^{-1}(\eta)\cdot(\eta-\eta_0)\le c_4(1+|\eta|^{p'})\quad\text{for all $\eta\in\R^{d\times d}_{sym}$},\label{eq:phins}
\end{align}
for a suitable constant $c_4>0$ independent of $n\in\N$.

For all $k\in\{1,\dots,n\}$ we search for a function $u_n^k\in V$ with $u_n^k-z_n^k\in V_n^k$ satisfying the following identity
\begin{equation}\label{unkM}
(\delta^2u_n^k,\varphi)_H+(G^{-1}_n(eu_n^k+e\delta u_n^k),e\varphi)_{p,p'}=(f_n^k,\varphi)_H\quad \text{for all $\varphi\in V_n^k$}, 
\end{equation}
where
\begin{equation}\label{derivate}
 \delta u_n^k:=\frac{u_n^k-u_n^{k-1}}{\tau_n},\qquad \delta^2 u_n^k:=\frac{\delta u_n^k-\delta u_n^{k-1}}{\tau_n}\quad\text{for $k\in\{1,\dots,n\}$}.
\end{equation}
To this aim, we find a function $v_n^k\in V_n^k$ which solves 
\begin{equation}\label{eqv}
(\delta^2 v_n^k+\delta^2 z_n^k,\varphi)_H+(G^{-1}_n(ev_n^k+e\delta v_n^k+ez_n^k+e\delta z_n^k),e\varphi)_{p,p'}=(f_n^k,\varphi)_H\quad \text{for all $\varphi\in V_n^k$}, 
\end{equation}
where $\delta z_n^k$ and $\delta^2z_n^k$ are defined similarly to~\eqref{derivate} starting from $z_n^k$. Indeed, the function $v_n^k\in V_n^k$ solves~\eqref{eqv} if and only if $u_n^k:=v_n^k+z_n^k\in V$ satisfies $u_n^k-z_n^k=v_n^k\in V_n^k$ and~\eqref{unkM}. 

To solve~\eqref{eqv}, we consider the family of nonlinear operators $F_n^k\colon V_n^k\rightarrow (V_n^k)'$ defined by
\begin{align*}
\langle F_n^k(v),\varphi \rangle_{(V_n^k)'}:=&\tfrac{1}{\tau_n^2}(v+v_n^{k-2}-2v_n^{k-1}+\tau_n^2\delta^2z_n^k-\tau_n^2f_n^k,\varphi)_H\\
&+(G^{-1}_n((1+\tfrac{1}{\tau_n})ev-\tfrac{1}{\tau_n}v_n^{k-1}+ez_n^k+e\delta z_n^k),e\varphi)_{p,p'}
\end{align*}
for $v,\varphi\in V_n^k$. It is clear that $v_n^k\in V_n^k$ solves~\eqref{eqv} if and only if 
\begin{equation}\label{eq:Fnk}
F_n^k(v_n^k)=0\quad\text{in $(V_n^k)'$}. 
\end{equation}
To find a solution to~\eqref{eq:Fnk} we need the following result, whose proof can be found in~\cite{Browder,Minty}.

\begin{theorem}[Browder-Minty]\label{thm:minty}
Let $X$ be a reflexive Banach space and let $F\colon X\rightarrow X'$ be a monotone, hemicontinuous, and coercive operator. Then $F$ is surjective. Moreover, if $F$ is strictly monotone, then $F$ is also injective. 
\end{theorem}

Let us show that $F_n^k$ satisfies the hypotheses of Theorem~\ref{thm:minty}.

\begin{proposition}\label{prop:FBM}
For every $n\in\N$ and $k\in\{1,\dots,n\}$ the nonlinear operator $F_n^k\colon V_n^k\rightarrow (V_n^k)'$ is strictly monotone, coercive, and hemicontinuous.
\end{proposition}
\begin{proof}
Let us fix $n\in\N$ and $k\in\{1,\dots,n\}$. We start by proving that $F_n^k$ is a strictly monotone operator, i.e.,
\begin{equation*}
 \langle F_n^k(v)-F_n^k(w),v-w\rangle_{(V_n^k)'}>0 \quad \text{for all $v,w\in V_n^k$ with $v\neq w$.}
\end{equation*}
By the definition of $F_n^k$, for all $v,w\in V_n^k$ with $v\neq w$ we have
\begin{equation}\label{corc}
 \langle F_n^k(v)-F_n^k(w),v-w\rangle_{(V_n^k)'}=\frac{1}{\tau_n^2}\|v-w\|^2_H+(G^{-1}_n(c_n ev+h_n^k)-G^{-1}_n(c_n ew+h_n^k),ev-ew)_{p,p'},
\end{equation}
where 
$$c_n:=1+\tfrac{1}{\tau_n}>0,\quad h_n^k:=-\tfrac{1}{\tau_n}ev_n^{k-1}+ez_n^k+e\delta z_n^k\in V_n^k.$$
By using in~\eqref{corc} the monotonicity of $G_n^{-1}$ with $\eta_1=c_nev+h_n^k$ and $\eta_2=c_n ew+h_n^k$, we can write 
\begin{equation*}
 \langle F_n^k(v)-F_n^k(w),v-w\rangle_{(V_n^k)'}\ge \frac{1}{\tau_n^2}\|v-w\|^2_H>0,
\end{equation*}
which shows the strictly monotonicity of $F_n^k$.

To prove the coerciveness of $F_n^k$, we have to show that
\begin{equation}\label{coer}
\frac{\langle F_n^k(v),v\rangle_{(V_n^k)'}}{\|v\|_{V_n^k}}\rightarrow \infty\quad \text{as $\|v\|_{V_n^k}\to \infty$}.
\end{equation}
Notice that
\begin{align*}
\langle F_n^k(v),v\rangle_{(V_n^k)'} = & d_n\|v\|_H^2 + d_n(\ell_n^k,v)_H \\
&+\frac{1}{c_n}(G^{-1}_n(c_n ev+h_n^k),c_n ev+h_n^k)_{p,p'}-\frac{1}{c_n}(G^{-1}_n(c_n ev+h_n^k),h_n^k)_{p,p'},
 \end{align*}
where 
$$d_n:=\frac{1}{\tau_n^2}>0,\quad\ell_n^k:=v_n^{k-2}-2v_n^{k-1}+\tau_n^2\delta^2z_n^k+\tau_n^2f_n^k\in H.$$ Thanks to~\eqref{eq:G2n},~\eqref{eq:G3n}, and Young inequality, for all $\varepsilon>0$ we have
\begin{align}\label{Gest}
&\frac{1}{c_n}(G^{-1}_n(c_n ev+h_n^k),c_n ev+h_n^k)_{p,p'}-\frac{1}{c_n}(G^{-1}_n(c_n ev+h_n^k),h_n^k)_{p,p'}\nonumber\\
 &\geq \frac{c_1}{c_n}\|c_n ev+h_n^k\|^{p'}_{p'}-\frac{c_2}{c_n}\mathcal L^d(\Omega)-\frac{1}{c_n}\|G^{-1}_n(c_n ev+h_n^k)\|_{p}\|h_n^k\|_{p'}\nonumber\\
 &\geq\frac{c_1}{c_n}\|c_n ev+h_n^k\|^{p'}_{p'}-\frac{c_2}{c_n}\mathcal L^d(\Omega)-\frac{\varepsilon^p}{pc_n}\|G^{-1}_n(c_n ev+h_n^k)\|^p_{p}-\frac{1}{p'c_n\varepsilon^{p'}}\|h_n^k\|^{p'}_{p'}\nonumber\\
 &\geq \frac{c_1}{c_n}\|c_n ev+h_n^k\|^{p'}_{p'}-\frac{c_2}{c_n}\mathcal L^d(\Omega)-\frac{1}{p'c_n\varepsilon^{p'}}\|h_n^k\|^{p'}_{p'}-\frac{\varepsilon^p}{pc_n}(2^{p-1}c_3^p\|c_n ev+h_n^k\|^{p'}_{p'}+2^{p-1}c_3^p \mathcal L^d(\Omega))\nonumber\\
 &=\frac{1}{c_n}\Big(c_1-\frac{2^{p-1}c_3^p\varepsilon^p}{p}\Big)\|c_n ev+h_n^k\|^{p'}_{p'}-\frac{1}{p'c_n\varepsilon^{p'}}\|h_n^k\|^{p'}_{p'}-\frac{1}{c_n}\Big(c_2+\frac{2^{p-1}c_3^p\varepsilon^p}{p} \Big)\mathcal L^d(\Omega).
\end{align}
In particular, the Korn-Poincaré inequality~\eqref{eq:kornpoi} yields
\begin{equation*}
 \frac{c_n^{p'}}{C_{KP}^{p'}}\|v\|_{V_n^k}^{p'}\le \|c_n ev\|^{p'}_{p'}\leq 2^{p'-1}\|c_n ev+h_n^k\|^{p'}_{p'}+2^{p'-1}\|h_n^k\|^{p'}_{p'}.
\end{equation*}
Hence, from~\eqref{Gest} we deduce
\begin{align}\label{Gnm}
&\frac{1}{c_n}(G^{-1}_n(c_n ev+h_n^k),c_n ev+h_n^k)_{p,p'}-\frac{1}{c_n}(G^{-1}_n(c_n ev+h_n^k),h_n^k)_{p,p'}\nonumber\\
&\geq \frac{c_n^{p'-1}}{2^{p'-1}C_{KP}^{p'}}\Big(c_1-\frac{2^{p-1}c_3^p\varepsilon^p }{p}\Big)\|v\|^{p'}_{V_n^k}-\frac{1}{c_n}\Big(c_1-\frac{2^{p-1}c_3^p\varepsilon^p }{p}+\frac{1}{p'\varepsilon^{p'}}\Big)\|h_n^k\|^{p'}_{p'}\nonumber\\
&\quad-\frac{1}{c_n}\Big(c_2+\frac{2^{p-1}c_3^p\varepsilon^p}{p} \Big)\mathcal L^d(\Omega).
\end{align}
By applying again Young inequality we can write
\begin{align}\label{l2}
 &d_n\|v\|_H^2+d_n(\ell_n^k,v)_H\geq \frac{d_n}{2}\|v\|_H^2-\frac{d_n}{2}\|\ell_n^k\|_H^2.
\end{align}
If we choose
\begin{equation*}
0<\varepsilon<\left(\frac{c_1 p}{2^{p-1}c_3^p}\right)^{\frac{1}{p}},
\end{equation*}
thanks to~\eqref{Gnm} and~\eqref{l2} we obtain the existence a positive constant $K_1$ such that
\begin{align}\label{interm}
\langle F_n^k(v),v\rangle_{(V_n^k)'} \geq K_1\left(\|v\|_H^2+\| v\|^{p'}_{V_n^k}-\|h_n^k\|^{p'}_{p'}-\|\ell_n^k\|_H^2-1\right).
\end{align}
Clearly, we have
\begin{equation}
\frac{\|h_n^k\|^{p'}_{p'}+\|\ell_n^k\|_H^2+1}{\|v\|_{V_n^k}}\rightarrow 0\quad \text{as $\|v\|_{V_n^k}\to \infty$}.
\end{equation}
Moreover, we can write
\begin{equation}\label{ordsup}
 \frac{\|v\|_H^2+\| v\|^{p'}_{V_n^k}}{\|v\|_{V_n^k}}\geq \|v\|_{V_n^k}^{p'-1}\rightarrow\infty\quad \text{as $\|v\|_{V_n^k}\to \infty$}.
\end{equation}
Thanks to~\eqref{interm}--\eqref{ordsup} we get~\eqref{coer}.

To prove the hemicontinuity of $F_n^k$ we need to show that for all $u,v,w\in V_n^k$ there exists $t_0=t_0(u,v,w)$ such that the function $[-t_0,t_0]\ni t\mapsto \langle F_n^k(v+tu),w\rangle_{(V_n^k)'}$ is continuous in $t=0$. We fix $u,v,w\in V_n^k$ and we notice that
\begin{equation*}
 \langle F_n^k(v+tu),w\rangle_{(V_n^k)'}= d_n(v+\ell_n^k,w)_H+d_nt(u,w)_H+(G_n^{-1}(c_n(ev+teu)+h_n^k),ew)_{p,p'}.
\end{equation*}
Moreover, we can write
\begin{equation}\label{gqo}
 G_n^{-1}(c_n(ev+teu)+h_n^k)\cdot ew\xrightarrow[t\to 0]{a.e.\ }G_n^{-1}(c_nev+h_n^k)\cdot ew,
\end{equation}
and thanks to~\eqref{eq:G3n} we get
\begin{align}\label{gesti}
|(G_n^{-1}(c_n(ev+teu)+h_n^k),ew)_{p,p'}|&\leq \frac{1}{p}\|G_n^{-1}(c_n(ev+teu)+h_n^k)\|_p^p+\frac{1}{p'}\|ew\|^{p'}_{p'}\nonumber\\
&\leq \frac{2^{p-1}c_3^p}{p}\|c_n(ev+teu)+h_n^k\|_{p'}^{p'}+\frac{2^{p-1}c_3^p}{p}\mathcal L^d(\Omega)+\frac{1}{p'}\|ew\|^{p'}_{p'}\nonumber\\
&\leq K_2(\|c_n ev+h_n^k\|_{p'}^{p'}+\|eu\|_{p'}^{p'}+\|ew\|^{p'}_{p'}+1),
\end{align}
for a positive constant $K_2$. By using~\eqref{gqo},~\eqref{gesti}, and dominate convergence theorem~we obtain
\begin{equation}\label{Gleb}
 (G_n^{-1}(c_n(ev+teu)+h_n^k), ew)_{p,p'}\xrightarrow[t\to 0]{}(G_n^{-1}(c_nev+h_n^k), ew)_{p,p'}.
\end{equation}
Since $d_nt(u,w)_H\rightarrow 0$ as $t\to 0$, by~\eqref{Gleb} we have
\begin{equation*}
 \langle F_n^k(v+tu),w\rangle_{(V_n^k)'}\xrightarrow[t\to 0]{}d_n(v+\ell_n^k,w)_H+(G_n^{-1}(c_n ev+h_n^k),ew)_{p,p'}=\langle F_n^k(v),w\rangle_{(V_n^k)'},
\end{equation*}
which concludes the proof.
\end{proof}

Thanks to Theorem~\ref{thm:minty} and Proposition~\ref{prop:FBM} we obtain that for all $n\in\N$ and $k\in\{1,\dots,n\}$ the nonlinear operator $F_n^k\colon V_n^k\rightarrow (V_n^k)'$ is bijective, and hence there exists a unique $v_n^k\in V_n^k$ which solves~\eqref{eqv}. As a consequence, the function $u_n^k=v_n^k+z_n^k\in V$ is the unique solution to~\eqref{unkM}.

Let us define
\begin{equation}\label{eq:sigmani}
\sigma_n^k:=G_n^{-1}(eu_n^k+e\delta u_n^k)\quad\text{for all $k\in\{1,\dots,n\}$}.    
\end{equation}
In the next lemma we show a uniform energy estimate with respect to $n$ for the family $\{(u_n^k,\sigma_n^k)\}_{k=1}^n$, which will be used to pass to the limit as $n\to\infty$ in the discrete equation~\eqref{unkM}.

\begin{lemma}\label{lem:estM}
There exists a positive constant $C_1$, independent of $n\in\N$, such that 
\begin{align}
\max_{i\in\{1,\dots,n\}}\|u_n^i\|_V+\max_{i\in\{1,\dots,n\}}\|\delta u_n^i\|_H+\sum_{i=1}^n \tau_n\|\delta u_n^i\|^{p'}_{V}+\sum_{i=1}^n\tau_n\|\sigma_n^i\|^{p}_{p}\le C_1\label{eq:estM}.
\end{align}
\end{lemma}

\begin{proof}
We take $\varphi=\tau_n(u_n^k-z_n^k)\in V_n^k$ as a test function in~\eqref{unkM}. Therefore, we obtain
\begin{align}\label{us}
&\tau_n(G^{-1}_n(eu_n^k+e\delta u_n^k),e u_n^k-ez_n^k)_{p,p'}=\tau_n(f_n^k, u_n^k- z_n^k)_H-\tau_n(\delta^2 u_n^k, u_n^k-z_n^k)_H.
\end{align}
We fix $i\in\{1,\dots,n\}$ and by summing in~\eqref{us} for $k\in\{1,\dots,i\}$ we obtain
\begin{align}\label{us3}
&\sum_{k=1}^i\tau_n(G^{-1}_n(eu_n^k+e\delta u_n^k),e u_n^k)_{p,p'}\nonumber\\
&=\sum_{k=1}^i\tau_n(G^{-1}_n(eu_n^k+e\delta u_n^k),e z_n^k)_{p,p'}+\sum_{k=1}^i\tau_n(f_n^k, u_n^k- z_n^k)_H-\sum_{k=1}^i\tau_n(\delta^2u_n^k, u_n^k- z_n^k)_H.
\end{align}
Now we use $\varphi=\tau_n(\delta u_n^k-\delta z_n^k)\in V_n^k$ as a test function in~\eqref{unkM} and we get
\begin{align}\label{us2}
&\|\delta u_n^k\|_H^2-(\delta u_n^{k-1},\delta u_n^k)_H+\tau_n(G^{-1}_n(eu_n^k+e\delta u_n^k), e \delta u_n^k)_{p,p'}\nonumber\\
&=\tau_n( f_n^k, \delta u_n^k-\delta z_n^k)_H+\tau_n(G^{-1}_n(eu_n^k+e\delta u_n^k),e \delta z_n^k)_{p,p'}+\tau_n (\delta^2 u_n^k, \delta z_n^k)_H.
\end{align}
By means of the following identity
\begin{align*}
\|\delta u_n^k\|_H^2-(\delta u_n^{k-1},\delta u_n^k)_H=\frac{1}{2}\|\delta u_n^k\|_H^2-\frac{1}{2}\|\delta u_n^{k-1}\|_H^2+\frac{\tau_n^2}{2}\|\delta^2 u_n^k\|_H^2,
\end{align*}
from~\eqref{us2} we infer
\begin{align*}
&\frac{1}{2}\|\delta u_n^k\|_H^2-\frac{1}{2}\|\delta u_n^{k-1}\|_H^2+\frac{\tau_n^2}{2}\|\delta^2 u_n^k\|_H^2+\tau_n(G^{-1}_n(eu_n^k+e\delta u_n^k),e \delta u_n^k)_{p,p'}\nonumber\\
&=\tau_n( f_n^k, \delta u_n^k-\delta z_n^k)_H+\tau_n(G^{-1}_n(eu_n^k+e\delta u_n^k),e \delta z_n^k)_{p,p'}+\tau_n (\delta^2 u_n^k, \delta z_n^k)_H,
\end{align*}
and, by summing again for $k\in\{1,\dots,i\}$ we get
\begin{align}\label{nmr2}
&\frac{1}{2}\|\delta u_n^i\|_H^2-\frac{1}{2}\|u^1\|_H^2+\sum_{k=1}^i\tau_n(G^{-1}_n(eu_n^k+e\delta u_n^k),e \delta u_n^k)_{p,p'}\nonumber\\
&\leq\sum_{k=1}^i\tau_n(f_n^k, \delta u_n^k-\delta z_n^k)_H+\sum_{k=1}^i\tau_n(G^{-1}_n(eu_n^k+e\delta u_n^k),e \delta z_n^k)_{p,p'}+\sum_{k=1}^i\tau_n (\delta^2 u_n^k, \delta z_n^k)_H.
\end{align}

By considering together~\eqref{us3} and~\eqref{nmr2} we get
\begin{align*}
&\frac{1}{2}\|\delta u_n^i\|_H^2+\sum_{k=1}^i\tau_n(G^{-1}_n(eu_n^k+e\delta u_n^k),eu_n^k+e \delta u_n^k)_{p,p'}\\
&\leq\frac{1}{2}\|u^1\|_H^2+\sum_{k=1}^i\tau_n(G^{-1}_n(eu_n^k+e\delta u_n^k),ez_n^k+e \delta z_n^k)_{p,p'}\\
&\quad+\sum_{k=1}^i\tau_n(f_n^k, u_n^k+\delta u_n^k- z_n^k-\delta z_n^k)_H+\sum_{k=1}^i\tau_n (\delta^2 u_n^k, z_n^k+ \delta z_n^k)_H-\sum_{k=1}^i\tau_n (\delta^2 u_n^k, u_n^k)_H.
\end{align*}
Thanks to~\eqref{eq:G1n}--\eqref{eq:G3n} and the Korn-Poincaré inequality~\eqref{eq:kornpoi} we deduce from the previous estimate
\begin{align}\label{grwM}
&\frac{1}{2}\|\delta u_n^i\|_H^2+\frac{c_1}{C_{KP}K^{p'}}\sum_{k=1}^i\tau_n\|u_n^k+\delta u_n^k\|^{p'}_V \nonumber\\
&\leq c_2T\mathcal L^d(\Omega)+\frac{1}{2}\|u^1\|_H^2+\sum_{k=1}^i\tau_n(G^{-1}_n(eu_n^k+e\delta u_n^k),ez_n^k+e \delta z_n^k)_{p,p'}\nonumber\\
&\quad+\sum_{k=1}^i\tau_n(f_n^k, u_n^k+\delta u_n^k- z_n^k-\delta z_n^k)_H+\sum_{k=1}^i\tau_n (\delta^2 u_n^k, z_n^k+\delta z_n^k)_H-\sum_{k=1}^i\tau_n (\delta^2 u_n^k, u_n^k)_H.
\end{align}

Let us now estimate the right-hand side of~\eqref{grwM} from above. We can write
\begin{align}
&\left|\sum_{k=1}^i \tau_n(f_n^k, u_n^k+\delta u_n^k)_H\right| 
\leq \|f\|^2_{L^2(0,T;H)}+\frac{1}{2}\sum_{k=1}^i \tau_n\|u_n^k\|_H^2+\frac{1}{2}\sum_{k=1}^i \tau_n\|\delta u_n^k\|_H^2,\\
&\left|\sum_{k=1}^i \tau_n( f_n^k, z_n^k+\delta z_n^k)_H\right| 
\leq \|f\|^2_{L^2(0,T;H)}+\frac{T}{2}\|z\|_{L^\infty(0,T;H)}^2+\frac{1}{2}\|\dot z\|_{L^2(0,T;H)}^2.
\end{align}
Moreover
\begin{align}
 &\left|\sum_{k=1}^i\tau_n(G^{-1}_n(eu_n^k+e\delta u_n^k),ez_n^k+e \delta z_n^k)_{p,p'}\right|\nonumber\\
&\leq\frac{\varepsilon^p}{p}\sum_{k=1}^i\tau_n\|G^{-1}_n(eu_n^k+e\delta u_n^k)\|_p^p+\frac{1}{p'\varepsilon^{p'}}\sum_{k=1}^i\tau_n\|ez_n^k+e\delta z_n^k\|_{p'}^{p'}\nonumber\\
&\leq\frac{2^{p-1}c^p_3\varepsilon^p}{p}\sum_{k=1}^i\tau_n\|u_n^k+\delta u_n^k\|_V^{p'}+\frac{2^{p-1}c^p_3T \varepsilon^p}{p}\mathcal L^d(\Omega)+\frac{2^{p'-1}T}{p'\varepsilon^{p'}}\|z\|_{L^\infty(0,T;V)}^{p'}+\frac{2^{p'-1}}{p'\varepsilon^{p'}}\|\dot z\|_{L^{p'}(0,T;V)}^{p'}.
\end{align}
Notice that the following discrete integration by parts formulas hold
\begin{align}
 &\sum_{k=1}^i \tau_n(\delta^2 u_n^k,z_n^k+\delta z^k_n)_H=(\delta u_n^i,z_n^i+\delta z_n^i)_H-(\delta u_n^0,z_n^0+\delta z_n^0)_H-\sum_{k=1}^i\tau_n (\delta u_n^{k-1},\delta z_n^k+\delta^2 z_n^k)_H,\label{dis-part}\\
 &\sum_{k=1}^i \tau_n(\delta^2 u_n^k,u_n^k)_H=(\delta u_n^i,u_n^i)_H-(\delta u_n^0,u_n^0)_H-\sum_{k=1}^i\tau_n (\delta u_n^{k-1},\delta u_n^k)_H.\label{dis-part2}
\end{align}
Since
\begin{equation}\label{ti2}
 \sum_{k=1}^i \tau_n \|\delta u_n^{k-1}\|_H^2=\sum_{k=0}^{i-1} \tau_n \|\delta u_n^k\|_H^2\leq T\|u^1\|_H^2+\sum_{k=1}^i \tau_n \|\delta u_n^k\|_H^2,
\end{equation}
thanks to~\eqref{dis-part} we can write for all $\varepsilon>0$
\begin{align}
 &\Big| \sum_{k=1}^i \tau_n(\delta^2 u_n^k,z_n^k+\delta z^k_n)_H\Big|\nonumber\\
 &\leq \frac{\varepsilon}{2}\|\delta u_n^i\|_H^2+\frac{1}{2\varepsilon}\|z_n^i+\delta z_n^i\|_H^2+\|u^1\|_H\|z(0)+\dot z(0)\|_H+\frac{1}{2} \sum_{k=1}^i \tau_n \|\delta u_n^{k-1}\|_H^2+\frac{1}{2} \sum_{k=1}^i \tau_n \|\delta z_n^k+\delta^2 z_n^k\|_H^2\nonumber\\
 &\leq K_1^\varepsilon+\frac{\varepsilon}{2}\|\delta u_n^i\|_H^2+\frac{1}{2}\sum_{k=1}^i \tau_n \|\delta u_n^k\|_H^2,
\end{align}
where $K_1^\varepsilon$ is a positive constant depending on $\varepsilon$. Moreover, since $u_n^i=\sum_{k=1}^i\tau_n\delta u_n^k+u^0$ for all $i\in\{1,\dots,n\}$, the discrete Hölder inequality gives us
\begin{equation}\label{fudev}
 \|u_n^i\|_H\leq \sum_{k=1}^i\tau_n\|\delta u_n^k\|_H+\|u^0\|_H\leq T^{\frac{1}{2}}\left(\sum_{k=1}^i\tau_n\|\delta u_n^k\|_H^2\right)^{\frac{1}{2}}+\|u^0\|_H.
\end{equation}
Hence from~\eqref{dis-part2},~\eqref{ti2}, and~\eqref{fudev} we deduce
\begin{align}
\left| \sum_{k=1}^i \tau_n(\delta^2 u_n^k,u_n^k)_H\right|&\leq \frac{\varepsilon}{2}\|\delta u_n^i\|_H^2+\frac{1}{2\varepsilon}\|u_n^i\|_H^2+\|u^1\|_H\|u^0\|_H+\frac{1}{2} \sum_{k=1}^i \tau_n \|\delta u_n^{k-1}\|_H^2+\frac{1}{2} \sum_{k=1}^i \tau_n \|\delta u_n^k\|_H^2\nonumber\\
 &\leq K_2^\varepsilon+\frac{\varepsilon}{2}\|\delta u_n^i\|_H^2+\left(1+\frac{T}{\varepsilon}\right)\sum_{k=1}^i \tau_n \|\delta u_n^k\|_H^2,
\end{align}
where $K_2^\varepsilon$ is a positive constant depending on $\varepsilon$. Furthermore
\begin{equation}\label{parts2}
\frac{1}{2}\sum_{k=1}^i \tau_n\|u_n^k\|_H^2\le T\|u^0\|_H^2+T^2\sum_{k=1}^i\tau_n\|\delta u_n^i\|_H^2.
\end{equation}
If we consider together~\eqref{grwM}--\eqref{parts2}, we get
\begin{align*}
\Big(\frac{1}{2}-\varepsilon\Big)\|\delta u_n^i\|_H^2&+\left(\frac{c_1}{C_{KP}^{p'}}-\frac{2^{p-1}c^p_3\varepsilon^p}{p}\right)\sum_{k=1}^i\tau_n\|u_n^k+\delta u_n^k\|^{p'}_{V}\leq K_3^\varepsilon\left(1+\sum_{k=1}^i\tau_n\|\delta u_n^k\|_H^2\right),
\end{align*}
where $K_3^\varepsilon$ is a positive constant depending on $\varepsilon$. By choosing
\begin{equation*}
 0<\varepsilon<\min\left\{\frac{1}{2},\Big(\frac{c_1 p}{C_{KP}^{p'}2^{p-1}c^p_3}\Big)^{\frac{1}{p}}\right\}
\end{equation*}
we get the existence of a positive constant $K_4$ independent of $n$ and $i$ such that
\begin{align}\label{gg}
\|\delta u_n^i\|_H^2+\sum_{k=1}^i\tau_n\|u_n^k+\delta u_n^k\|^{p'}_{V}\leq K_4\left(1+\sum_{k=1}^i\tau_n\|\delta u_n^k\|_H^2\right).
\end{align}
By defining $a_n^i:=\|\delta u_n^i\|_H^2$ for all $i\in\{1,\dots,n\}$, from~\eqref{gg} we derive
$$a_n^i\leq K_4\left(1+\sum_{k=1}^i \tau_n a_n^k\right)\quad\text{for all $i\in\{1,\dots,n\}$},$$
and taking into account a discrete version of Gronwall lemma (see, e.g.,~\cite[Lemma~3.2.4]{AGS}) we deduce that the family $\{a_n^i\}_{i=1}^n$ is bounded by a positive constant $K_5$ independent of $i$ and $n$, i.e.,
\begin{equation}\label{vel}
\|\delta u_n^i\|^2\leq K_5\quad \text{for all $i\in\{1,\dots,n\}$ and $n\in\N$}.
\end{equation}
By using~\eqref{gg} and~\eqref{vel} we get the existence of a positive constant $K_6$ independent of $n$ such that 
\begin{equation}\label{mezz}
\max_{i\in\{1,\dots,n\}}\|\delta u_n^i\|_H^2+\sum_{i=1}^n \tau_n \|u_n^i+\delta u_n^i\|^{p'}_{V}\le K_6.
\end{equation}
In particular, by~\eqref{eq:G3n} and~\eqref{eq:sigmani} it holds
\begin{equation}\label{sigmank}
\sum_{i=1}^n\tau_n\|\sigma_n^i\|_p^p\le 2^{p-1}c_3^p\sum_{i=1}^n\tau_n\|eu_n^i+e\delta u_n^i\|_{p'}^{p'}+2^{p-1}c_3^p T\mathcal L^d(\Omega)\le 2^{p-1}c_3^pK_6+2^{p-1}c_3^p T\mathcal L^d(\Omega).
\end{equation}

To get the last estimate in~\eqref{eq:estM} we set $b_n^k:=(1+\tau_n)^k$ for $k\in\{0,\dots,n\}$ and we notice that 
\begin{equation}\label{reg}
 \frac{b_n^k-b_n^{k-1}}{\tau_n}=b_n^{k-1}\quad\text{for all $k\in\{1,\dots,n\}$}.
\end{equation}
From~\eqref{reg} we can write
\begin{align}\label{new}
b_n^ku_n^k-b^0u^0&=\sum_{j=1}^k(b_n^ju_n^j-b_n^{j-1}u_n^{j-1})\nonumber\\
&=\sum_{j=1}^k\tau_n \frac{b_n^j-b_n^{j-1}}{\tau_n}u_n^j+\sum_{j=1}^k\tau_n b_n^{j-1} \frac{u_n^j-u_n^{j-1}}{\tau_n}=\sum_{j=1}^k\tau_n b_n^{j-1}(u_n^j+\delta u_n^j).
\end{align}
Since
\begin{equation*}
 1\leq (1+\tau_n)^k\leq (1+\tau_n)^n=\left[\left(1+\frac{T}{n}\right)^{\frac{n}{T}}\right]^T\leq \e^T,
\end{equation*}
from~\eqref{new} we deduce the existence of a positive constant $K_7$ such that
\begin{equation}\label{bunk}
 \|u_n^k\|^{p'}_{V}\leq K_7\left(1+\sum_{j=1}^k\tau_n \|u_n^j+\delta u_n^j\|^{p'}_{V}\right)\leq K_7\left(1+K_6\right).
\end{equation}
As a consequence of this, we obtain
\begin{equation}\label{bdunk}
 \sum_{j=1}^k\tau_n \|\delta u_n^j\|^{p'}_{V}\leq 2^{p'-1}\sum_{j=1}^k\tau_n (\|u_n^j+\delta u_n^j\|^{p'}_{V}+\|u_n^j\|^{p'}_{V})\leq 2^{p'-1}\left(K_6+TK_7+TK_7K_6\right).
\end{equation}
Hence by considering together~\eqref{mezz},~\eqref{sigmank},~\eqref{bunk}, and~\eqref{bdunk} we get~\eqref{eq:estM}.
\end{proof}

As a consequence of~\eqref{eq:estM} and of the particular form of equation~\eqref{unkM}, we derive also a uniform bound on the discrete time second derivative of $u_n^k$ in the space $H$. This allow us to find in the limit as $n\to\infty$ a weak solution to~\eqref{eq:nonlin-KV}--\eqref{eq:boundaryN} with displacement $u\in W^{2,2}(0,T;H)$.

\begin{corollary}
There exists a constant $C_2$, independent of $n\in\N$, such that 
\begin{equation}\label{sec-der-bound}
\sum_{k=1}^n\tau_n\|\delta^2 u^k_n\|_H^2\leq C_2.
\end{equation}
\end{corollary}

\begin{proof}
Let us define $v_n^k:=u_n^k+\delta u_n^k\in V$ for all $k\in\{1,\dots,n\}$ and $n\in\N$. By equation~\eqref{unkM} we deduce that $v_n^k$ solves
the following equation
\begin{equation*}
(\delta v_n^k,\varphi)_H+(G^{-1}_n(ev_n^k),e\varphi)_{p,p'}=(f_n^k+\delta u_n^k,\varphi)_H\quad \text{for all $\varphi\in V_n^k$}. 
\end{equation*}
We take $\varphi:=\tau_n(\delta v_n^k-\delta z_n^k-\delta^2 z_n^k)\in V_n^k$ as a test function in~\eqref{unkM}. We fix $i\in\{1,\dots,n\}$ and by summing over $k=1,\dots i$ we get
\begin{align}\label{eq:vnk_est}
&\sum_{k=1}^i\tau_n\|\delta v_n^k\|_H^2+\sum_{k=1}^i(G^{-1}_n(ev_n^k),e v_n^k-ev_n^{k-1})_{p,p'}\nonumber\\
&=\sum_{k=1}^i\tau_n(f_n^k+\delta u_n^k,\delta v_n^k)_H-\sum_{k=1}^i\tau_n(f_n^k+\delta u_n^k,\delta z_n^k+\delta^2 z_n^k)_H\nonumber\\
&\quad +\sum_{k=1}^i\tau_n (\delta v_n^k,\delta z_n^k+\delta^2 z_n^k)_H+\sum_{k=1}^i\tau_n(G^{-1}_n(ev_n^k),e\delta z_n^k+e\delta^2 z_n^k)_{p,p'}.
\end{align}
 
Let us now estimate the right-hand side of~\eqref{eq:vnk_est} from above. Thanks to~\eqref{eq:estM} we can write
\begin{align}
&\left|\sum_{k=1}^i \tau_n(f_n^k+\delta u_n^k, \delta v_n^k)_H\right| 
\leq \frac{1}{2\varepsilon}\|f\|^2_{L^2(0,T;H)}+\frac{T C_1^2}{2\varepsilon}+\varepsilon\sum_{k=1}^i \tau_n\|\delta v_n^k\|^2_H,\\
&\left|\sum_{k=1}^i \tau_n(f_n^k+\delta u_n^k, \delta z_n^k+\delta^2 z_n^k)_H\right| 
\leq \|f\|^2_{L^2(0,T;H)}+TC_1^2+\|\dot z\|^2_{W^{1,2}(0,T;H)},\\
&\left|\sum_{k=1}^i\tau_n (\delta v_n^k,\delta z_n^k+\delta^2 z_n^k)_H\right|
\leq \varepsilon\sum_{k=1}^i \tau_n\|\delta v_n^k\|^{2}_{H}+\frac{1}{2\varepsilon}\|\dot z\|_{W^{1,2}(0,T;H)}^2.
\end{align}
Moreover
\begin{align}
&\left|\sum_{k=1}^i\tau_n(G^{-1}_n(ev_n^k),e\delta z_n^k+e \delta^2 z_n^k)_{p,p'}\right|\nonumber\\
&\leq\frac{1}{p}\sum_{k=1}^i\tau_n\|G^{-1}_n(ev_n^k)\|_p^p+\frac{1}{p'}\sum_{k=1}^i\tau_n\|e\delta z_n^k+e\delta^2 z_n^k\|_{p'}^{p'}\nonumber\\
&\leq \frac{2^{p-1}c_3^p}{p}\sum_{k=1}^i\tau_n\|u_n^k+\delta u_n^k\|_V^{p'}+\frac{2^{p-1}c_3^pT}{p}\mathcal L^d(\Omega)+\frac{2^{p-1}}{p'}\|\dot z\|_{W^{1,p'}(0;T;V)}^{p'}\nonumber\\
&\leq \frac{4^{p-1}c_3^p}{p}(TC_1^p+C_1)+\frac{2^{p-1}c_3^pT}{p}\mathcal L^d(\Omega)+\frac{2^{p-1}}{p'}\|\dot z\|_{W^{1,p'}(0;T;V)}^{p'}.
\end{align}
Finally, by~\eqref{eq:G1n} and the convexity of $\phi_n^*$ we have
\begin{align}\label{eq:Gvnk}
\sum_{k=1}^i(G^{-1}_n(ev_n^k),e v_n^k-ev_n^{k-1})_{p,p'}&\ge \sum_{k=1}^i \left(\int_\Omega\phi_n^*(ev_n^k(x))\,\de x-\int_\Omega\phi_n^*(ev_n^{k-1}(x))\,\de x\right)\nonumber\\
&=\int_\Omega\phi_n^*(ev_n^i(x))\,\de x-\int_\Omega\phi_n^*(ev_n^0(x))\,\de x.
\end{align}
By combining~\eqref{eq:vnk_est}--\eqref{eq:Gvnk} with the bound~\eqref{eq:phins} for $\phi_n^*$, we deduce the existence of a positive constant $K^\varepsilon$, which depends on $\varepsilon$, but it is independent of $n$ and $i$, such that
\begin{equation}\label{eq:d2unk}
(1-2\varepsilon)\sum_{k=1}^i\tau_n\|\delta v_n^k\|^2_H+\int_\Omega\phi_n^*(ev_n^i(x))\,\de x\le K^\varepsilon\quad\text{for all $i\in\{1,\dots,n\}$}.
\end{equation}
By choosing $\varepsilon=\frac{1}{2}$ and using~\eqref{eq:phin0}, from~\eqref{eq:estM} and~\eqref{eq:d2unk} we deduce~\eqref{sec-der-bound}.
\end{proof}

We now want to pass to the limit as $n\to\infty$ in the discrete equation~\eqref{unkM} to obtain a weak solution $(u,\sigma)$ to the nonlinear viscoelastic system~\eqref{eq:nonlin-KV}--\eqref{eq:boundaryN}, according to Definition~\ref{def:weak_sol}. We start by defining the following interpolation sequences of $\{(u_n^k,\sigma_n^k)\}_{k=1}^n$:
\begin{align*}
&u_n(t):=u_n^k+(t-k\tau_n)\delta u_n^k, &&\tilde{u}_n(t):=\delta u_n^k+(t-k\tau_n)\delta^2 u_n^k, && t\in [(k-1)\tau_n,k\tau_n],\quad k\in\{1,\dots,n\},\\
&u^+_n(t):=u_n^k, &&\tilde{u}^+_n(t):=\delta u_n^k, && t\in ((k-1)\tau_n,k\tau_n], \quad k\in\{1,\dots,n\},\\
&u^+_n(0):=u_n^0=u^0, &&\tilde{u}^+_n(t):=\delta u_n^0=u^1, &&\\
&u^-_n(t):=u_n^{k-1}, &&\tilde{u}^-_n(t):=\delta u_n^{k-1},&& t\in [(k-1)\tau_n,k\tau_n), \quad k\in\{1,\dots,n\},\\
&u^-_n(T):=u_n^n, &&\tilde{u}^-_n(T):=\delta u_n^n,&& \\
& \sigma_n^+(t):=\sigma_n^k,&& && t\in ((k-1)\tau_n,k\tau_n], \quad k\in\{1,\dots,n\}.
\end{align*}
By means of this notation, we can state the following convergence lemma.

\begin{lemma}\label{lem:convKVF}
There exists a pair $(u,\sigma)\in (\V\cap W^{2,2}(0,T;H))\times L^p(0,T;L^p(\Omega;\R^{d\times d}_{sym}))$ such that, up to a not relabeled subsequence
\begin{alignat}{4}
&u_n\xrightharpoonup[n\to \infty]{W^{1,p'}(0, T;V)}u,&&\quad u_n\xstararrow{n\to \infty}{W^{1,\infty}(0, T;H)}u, &&\quad \tilde u_n\xrightharpoonup[n\to \infty]{L^{p'}(0, T;V)}\dot u, &&\quad \tilde u_n \xrightharpoonup[n\to \infty]{W^{1,2}(0, T;H)}\dot u, \label{weak-conv}\\
& u^\pm_n \xstararrow{n\to \infty}{L^\infty(0, T;V)}u,&&\quad \tilde{u}^\pm_n \xrightharpoonup[n\to \infty]{L^{p'}(0, T;V)}\dot{u} ,&&\quad \tilde{u}^\pm_n \xstararrow{n\to \infty}{L^\infty(0, T;H)}\dot{u},&&\quad \sigma_n^+\xrightharpoonup[n\to\infty]{L^p(0,T;L^p(\Omega;\R^{d\times d}_{sym}))}\sigma.\label{weak-conv2}
\end{alignat}
Moreover
\begin{align}
&u_n(t)\xrightharpoonup[n\to\infty]{V}u(t), & &\tilde u_n(t)\xrightharpoonup[n\to\infty]H\dot u(t)& &\text{for all $t\in[0,T]$,}\label{point-conv}\\
&u_n^\pm(t)\xrightharpoonup[n\to\infty]{V}u(t), & &\tilde u_n^\pm(t)\xrightharpoonup[n\to\infty]H\dot u(t)& &\text{for all $t\in[0,T]$.}\label{point-conv2}
\end{align}
\end{lemma}

\begin{proof}
Thanks to Lemma~\ref{lem:estM} and the estimate~\eqref{sec-der-bound}, the sequences 
\begin{align*}
&\{u_n\}_n\subset W^{1,p'}(0, T;V)\cap W^{1,\infty}(0, T;H),
& &\{\tilde u_n\}_n\subset L^{p'}(0,T;V)\cap W^{1,2}(0,T;H),\\
& \{\sigma_n^+\}_n\subset L^p(0,T;L^p(\Omega;\R^{d\times d}_{sym})),
\end{align*}
are uniformly bounded with respect to $n\in\N$. Indeed, we have
\begin{align*}
&\|u_n\|^{p'}_{W^{1,p'}(0, T;V)}\leq T\max_{k\in\{0,\dots,n\}}\|u_n^k\|_V^{p'}+\sum_{k=1}^n\tau_n\|\delta u_n^k\|^{p'}_V,\\
&\|u_n\|_{W^{1,\infty}(0, T;H)}\leq \max_{k\in\{0,\dots,n\}}\|u_n^k\|_H+\max_{k\in\{1,\dots,n\}}\|\delta u_n^k\|_H,\\
&\|\tilde{u}_n\|_{L^{p'}(0,T;V)}^{p'}\le 2\sum_{k=1}^n\tau_n\|\delta u_n^k\|_V^{p'}+\|u^1\|_V^{p'},\\
&\|\tilde{u}_n\|_{W^{1,2}(0,T;H)}^2\leq T\max_{k\in\{0,\dots,n\}}\|\delta u_n^k\|_H^2+\sum_{k=1}^n\tau_n\|\delta^2 u_n^k\|^2_H,\\
&\|\sigma_n^+\|^p_{L^p(0,T;L^p(\Omega;\R^{d\times d}_{sym}))}= \sum_{k=1}^n\tau_n\|\sigma_n^k\|_p^p.
\end{align*} 
By Banach-Alaoglu theorem and Lemma~\ref{lem:wcM} there exist three functions $u\in W^{1,p'}(0, T;V)\cap W^{1,\infty}(0, T;H)$, $v\in L^{p'}(0,T;V)\cap W^{1,2}(0,T;H)$, and $\sigma\in L^p(0,T;L^p(\Omega;\R^{d\times d}_{sym}))$ such that, up to a not relabeled subsequence
\begin{align}
&u_n\xrightharpoonup[n\to \infty]{W^{1,p'}(0, T;V)}u,&&\quad u_n\xstararrow{n\to \infty}{W^{1,\infty}(0, T;H)}u, &&\quad \tilde u_n\xrightharpoonup[n\to \infty]{L^{p'}(0, T;V)}v, &&\quad \tilde u_n \xrightharpoonup[n\to \infty]{W^{1,2}(0, T;H)}v, \label{weak-conv-1}
\end{align}
and 
\begin{align*}
&\sigma_n^+ \xrightharpoonup[n\to\infty]{L^p(0,T;L^p(\Omega;\R^{d\times d}_{sym}))}\sigma.
\end{align*}
Thanks to~\eqref{sec-der-bound} we get
\begin{align*}
\|\dot u_n-\tilde u_n\|_{L^2(0,T;H)}^2\leq \tau_n^2\sum_{k=1}^n\tau_n\|\delta^2 u_n^k\|^2_H\leq C_2\tau_n^2\xrightarrow[n\to\infty]{}0,
\end{align*}
from which we deduce that $v=\dot u$. 

By~\eqref{eq:estM} also the sequences
\begin{align}\label{b3}
&\{u_n^\pm\}_n\subset L^\infty(0, T;V),
& &\{\tilde u_n^\pm\}_n\subset L^{p'}(0,T;V)\cap L^\infty(0,T;H),
\end{align}
are uniformly bounded. Moreover, by using again~\eqref{eq:estM} and~\eqref{sec-der-bound} we have 
\begin{align*}
&\|u_n-u^{+}_n\|_{L^\infty(0, T;H)} \leq \tau_n\max_{k\in\{1,\dots,n\}}\ \|\delta u_n^k\|_H\leq C_1\tau_n\xrightarrow[n\to\infty]{}0,\\
&\|u^+_n-u^{-}_n\|_{L^\infty(0, T;H)}\leq \tau_n \max_{k\in\{1,\dots,n\}}\ \|\delta u_n^k\|_H\leq C_1\tau_n\xrightarrow[n\to\infty]{}0,\\
&\|\tilde u_n-\tilde u_n^+\|_{L^2(0,T;H)}^2\leq \tau_n^2\sum_{k=1}^n\tau_n\|\delta^2 u_n^k\|^2_{H}\leq C_2 \tau_n^2\xrightarrow[n\to\infty]{}0,\\
&\|\tilde u^+_n-\tilde u_n^-\|_{L^2(0,T;H)}^2\leq \tau_n^2\sum_{k=1}^n\tau_n\|\delta^2 u_n^k\|^2_H\leq C_2 \tau_n^2\xrightarrow[n\to\infty]{}0.
\end{align*}
We combine~\eqref{weak-conv-1} and~\eqref{b3} with the previous convergences to derive
\begin{equation*}
u^\pm_n \xstararrow{n\to \infty}{L^\infty(0, T;V)} u,\qquad \tilde{u}^\pm_n \xstararrow{n\to \infty}{L^\infty(0, T;H)}\dot{u},\qquad \tilde{u}^\pm_n \xrightharpoonup[n\to \infty]{L^{p'}(0, T;V)}\dot{u}.
\end{equation*}

Finally, by~\eqref{weak-conv-1} for all $t\in[0,T]$ we have
\begin{equation*}
u_n(t)\xrightharpoonup[n\to\infty]{V}u(t),\qquad \tilde u_n(t)\xrightharpoonup[n\to\infty]{H}\dot u(t).
\end{equation*}
Thanks to~\eqref{eq:estM} and~\eqref{sec-der-bound}, for all $t\in[0,T]$ we get
\begin{alignat*}{2}
&\|u_n^\pm(t)\|_{V}\le C_1,&&\quad\|u_n^+(t)-u_n(t)\|_H\leq C_1\tau_n\xrightarrow[n\to\infty]{}0,\\
& &&\quad\|u_n^+(t)-u^-_n(t)\|_H\leq C_1\tau_n\xrightarrow[n\to\infty]{}0,\\
&\|\tilde u_n^\pm(t)\|_H\le C_1,&&\quad\|\tilde u_n^+(t)-\tilde u_n(t)\|_{H}^2\leq \tau_n\sum_{k=1}^n\tau_n\|\delta^2 u_n^k\|^2_H\leq C_2\tau_n\xrightarrow[n\to\infty]{}0,\\
& &&\quad\|\tilde u_n^+(t)-\tilde u^-_n(t)\|_H^2=\tau_n\sum_{k=1}^n\tau_n\|\delta^2 u_n^k\|_H^2\leq C_2\tau_n\xrightarrow[n\to\infty]{}0,
\end{alignat*}
which imply~\eqref{point-conv} and~\eqref{point-conv2}. 
\end{proof}

In view of the compactness of the embedding $V\hookrightarrow H$ (see Remark~\ref{rem:comp}), we deduce also the following strong convergences.

\begin{corollary}
Let  $(u,\sigma)\in (\V\cap W^{2,2}(0,T;H))\times L^p(0,T;L^p(\Omega;\R^{d\times d}_{sym}))$ be the pair of function given by Lemma~\ref{lem:convKVF}. Then, we have
\begin{equation}\label{eq:strong-conv}
u_n^+\xrightarrow[n\to\infty]{L^2(0,T;H)} u,\quad \tilde u_n^+\xrightarrow[n\to\infty]{L^2(0,T;H)} \dot u.
\end{equation}
\end{corollary}

\begin{proof}
By Lemma~\ref{lem:convKVF} we know that the following sequences
\begin{align*}
&\{u_n\}_n\subset W^{1,p'}(0, T;V)\cap W^{1,\infty}(0, T;H),
& &\{\tilde u_n\}_n\subset L^{p'}(0,T;V)\cap W^{1,2}(0,T;H),
\end{align*}
are uniformly bounded with respect to $n$. Since the embedding $V\hookrightarrow H$ is compact, by Aubin-Lions lemma (see for example~\cite[Theorem 3]{Simon}), we derive
\begin{equation*}
u_n\xrightarrow[n\to\infty]{L^2(0,T;H)} u,\quad \tilde u_n\xrightarrow[n\to\infty]{L^2(0,T;H)} \dot u.
\end{equation*}
Moreover, we have
\begin{align*}
&\|u_n-u^+_n\|_{L^2(0, T;H)}^2 \leq \tau_n^2\sum_{k=1}^n\tau_n\|\delta u_n^k\|^2_H\leq TC_1\tau_n^2\xrightarrow[n\to\infty]{}0,\\
&\|\tilde u_n-\tilde u_n^+\|_{L^2(0,T;H)}^2\leq \tau_n^2\sum_{k=1}^n\tau_n\|\delta^2 u_n^k\|^2_{H}\leq C_2 \tau_n^2\xrightarrow[n\to\infty]{}0,
\end{align*}
which imply~\eqref{eq:strong-conv}.
\end{proof}

We want to prove that the pair $(u,\sigma)\in (\V\cap W^{2,2}(0,T;H))\times L^p(0,T;L^p(\Omega;\R^{d\times d}_{sym}))$ of Lemma~\ref{lem:convKVF} is a weak solution to the nonlinear viscoelastic system~\eqref{eq:nonlin-KV}--\eqref{eq:boundaryN} with initial conditions~\eqref{eq:initial}. To this aim, we need to check $(i)$--$(iii)$ of Definition~\ref{def:weak_sol} and that $u(0)=u^0$ in $V$ and $\dot u(0)=u^1$ in $H$.
We start by showing that the function $u\in \mathcal V\cap W^{2,2}(0,T;H)$ satisfies the Dirichlet boundary conditions and the initial conditions.

\begin{lemma}\label{lem:boun-con}
The function $u\in\V\cap W^{2,2}(0,T;H)$ of Lemma~\ref{lem:convKVF} satisfies $(i)$ of Definition~\ref{def:weak_sol} and the initial conditions $u(0)=u^0$ in $V$ and $\dot u(0)=u^1$ in $H$. 
\end{lemma}

\begin{proof}
By~\eqref{point-conv} we have
\begin{equation*}
u^0=u_n(0)\xrightharpoonup[n\to\infty]{V}u(0),\quad u^1=\tilde u_n(0)\xrightharpoonup[n\to\infty]H\dot u(0).
\end{equation*}
Hence, $u(0)=u^0$ in $V$ and $\dot u(0)=u^1$ in $H$. Moreover, since $z\in C^0([0,T];V_0)$ and thanks to~\eqref{point-conv2}, we have for all $t\in[0,T]$ 
\begin{equation*}
u_n^-(t)-z_n^-(t)\in V_t^D,\quad u_n^-(t)-z_n^-(t)\xrightharpoonup[n\to\infty]{V}u(t)-z(t).
\end{equation*}
Thus, $u(t)-z(t)\in V_t^D$ for all $t\in[0,T]$, being $V_t^D$ a closed subspace of $V$.
\end{proof} 

With the next lemma we show that the pair $(u,\sigma)$ solves the weak formulation~\eqref{eq:wweak} of the elastodynamics system.

\begin{lemma}\label{lem:wweak}
The pair $(u,\sigma)\in (\V\cap W^{2,2}(0,T;H))\times L^{p}(0,T;L^p(\Omega;\R^{d\times d}_{sym}))$ of Lemma~\ref{lem:convKVF} satisfies $(ii)$ of Definition~\ref{def:weak_sol}.
\end{lemma}

\begin{proof}
We fix $n\in\N$ and a function $\varphi\in \D$. We consider the following functions
\begin{alignat*}{3}
\varphi_n^k&:=\varphi(k\tau_n)\quad\text{for $k\in\{0,\dots,n\}$,}\quad\delta \varphi_n^k&:=\frac{\varphi_n^k-\varphi_n^{k-1}}{\tau_n}\quad \text{for }k\in\{1,\dots,n\},
\end{alignat*}
and the piecewise-constant approximating sequences
\begin{align*}
&\varphi^+_n(t):=\varphi_n^k, & & \tilde\varphi^+_n(t):=\delta\varphi_n^k, & & f^+_n(t):=f_n^k,& &\text{for } t\in ((k-1)\tau_n,k\tau_n],\quad k\in\{1,\dots,n\}.
\end{align*}
If we use $\tau_n\varphi_n^k\in V_n^k$ as a test function in~\eqref{unkM}, after summing over $k\in\{1,\dots,n\}$, we get
\begin{align}\label{limitM}
\sum_{k=1}^n\tau_n(\delta^2u_n^k,\varphi^k_n)_H+\sum_{k=1}^n\tau_n(\sigma_n^k,e\varphi^k_n)_{p,p'}=\sum_{k=1}^n\tau_n(f_n^k,\varphi^k_n)_H.
\end{align}
Since $\varphi_n^0=\varphi_n^n=0$ we obtain
\begin{align*}
\sum_{k=1}^n \tau_n(\delta^2 u^k_n,\varphi^k_n)_H
&=\sum_{k=1}^{n} (\delta u^k_n,\varphi^k_n)_H-\sum_{k=1}^n (\delta u^{k-1}_n,\varphi^k_n)_H=\sum_{k=0}^{n-1} (\delta u^k_n,\varphi^k_n)_H-\sum_{k=0}^{n-1} (\delta u^k_n,\varphi^{k+1}_n)_H\\
&=-\sum_{k=0}^{n-1} \tau_n (\delta u^k_n,\delta \varphi^{k+1}_n)_H=-\sum_{k=1}^n\tau_n(\delta u^{k-1}_n,\delta \varphi^k_n)_H=-\int_0^T(\tilde{u}^-_n(t),\tilde{\varphi}^+_n(t))_H\, \de t,
\end{align*}
and from~\eqref{limitM} we deduce
\begin{align}\label{eqappM}
-\int_0^T(\tilde{u}^-_n(t),\tilde{\varphi}^+_n(t))_H\,\de t+\int_0^T(\sigma_n^+(t),e \varphi^+_n(t))_{p,p'}\,\de t=\int_0^T( f^+_n(t),\varphi^+_n(t))_H\,\de t.
\end{align}
Thanks to~\eqref{weak-conv2} and the convergences 
\begin{equation*}
\varphi^+_n\xrightarrow[n\to\infty]{L^{p'}(0,T;V)}\varphi, \quad \varphi^+_n\xrightarrow[n\to\infty]{ L^2(0,T;H)}\varphi, \quad \tilde{\varphi}^+_n\xrightarrow[n\to\infty]{L^2(0,T;H)}\dot{\varphi},\quad f^+_n\xrightarrow[n\to\infty]{L^2(0,T;H)}f
\end{equation*}
we can pass to the limit in~\eqref{eqappM}, and we get that the pair $(u,\sigma)\in\V\times L^p(0,T;L^p(\Omega;\R^{d\times d}_{sym}))$ satisfies $(ii)$ of Definition~\ref{def:weak_sol}.
\end{proof}

Finally, we have that the pair $(u,\sigma)$ satisfies the constitutive law~\eqref{eq:conlaw}.

\begin{lemma}\label{lem:conlaw}
The pair $(u,\sigma)\in (\V\cap W^{2,2}(0,T;H))\times L^{p}(0,T;L^p(\Omega;\R^{d\times d}_{sym}))$ of Lemma~\ref{lem:convKVF} satisfies $(iii)$ of Definition~\ref{def:weak_sol}.
\end{lemma}

\begin{proof}
In order to verify the constitutive law, we use a modification of Minty method, as done in~\cite{BuPaSuSe,Patel}. Since $\dot u\in W^{1,2}(0,T;H)$, by integrating by parts in~\eqref{eq:wweak} we deduce that $(u,\sigma)$ solve
\begin{equation}\label{eq:cont}
\int_0^T(\ddot u(t),\varphi(t))_H\,\de t+\int_0^T(\sigma(t),e \varphi(t))_{p,p'}\,\de t=\int_0^T(f(t),\varphi (t))_H\,\de t\quad\text{for all $\varphi\in \D$}.
\end{equation}

Let us now consider a function $\varphi\in L^{p'}(0,T;V)\cap L^2(0,T;H)$ with $\varphi(t)\in V_t^D$ for a.e.\ $t\in[0,T]$. Then there exists a sequence of functions $\{\varphi_n\}_n\subset \mathcal D$ such that
\begin{equation*}
\varphi_n\xrightarrow[n\to\infty]{L^{p'}(0,T;V)}\varphi,\quad \varphi_n\xrightarrow[n\to\infty]{L^2(0,T;H)}\varphi.
\end{equation*}
This can be done, for example, by considering a sequence $\{\omega_n\}_n\subset C_c^1((\frac{2}{n},T-\frac{2}{n}))$ with $0\le \omega_n\le 1$ in $[0,T]$ for all $n\in\N$ and such that $\omega_n(t)\to 1$ as $n\to\infty$ for all $t\in(0,T)$, and a sequence $\{\rho_n\}_n \subset C_c^1((0,\frac{1}{n}))$ with $\rho_n\ge 0$ and $\int_{\R}\rho_n\,\de t=1$ for all $n\in\N$, and defining
$$\varphi_n:=\rho_n*(\omega_n \varphi)\quad\text{for all $n\in\N$}$$
(see also~\cite[Lemma 2.8]{DMT}). By testing~\eqref{eq:cont} with $\varphi_n$ and passing to the limit as $n\to\infty$ can deduce that the pair $(u,\sigma)$ satisfies
\begin{equation}\label{eq:strong_form}
\int_0^T(\ddot u(t),\varphi(t))_H\,\de t+\int_0^T(\sigma(t),e \varphi(t))_{p,p'}\,\de t=\int_0^T(f(t),\varphi (t))_H\,\de t
\end{equation}
for all $\varphi\in L^{p'}(0,T;V)\cap L^2(0,T;H)$ with $\varphi(t)\in V_t^D$ for a.e.\ $t\in[0,T]$. Notice that 
$$\dot u-\dot z\in L^{p'}(0,T;V)\cap L^2(0,T;H),\quad\varphi(t)\in V_t^D\quad\text{for a.e.\ $t\in[0,T]$},$$
since $\frac{u(t)-u(t-h)}{h}-\frac{z(t)-z(t-h)}{h}\in V_t^D$ for all $t\in(0,T]$ and $h\in (0,t)$, and 
$$\frac{u(t)-u(t-h)}{h}-\frac{z(t)-z(t-h)}{h}\to \dot u(t)-\dot z(t)\quad\text{for a.e.\ $t\in[0,T]$ as $h\to 0$}.$$
Hence, by using $\varphi:=u+\dot u-z-\dot z$ in~\eqref{eq:strong_form} we get
\begin{align}
\int_0^T(\sigma(t),e u(t)+e\dot u(t))_{p,p'}\,\de t&=\int_0^T(f(t),u (t)+\dot u(t))_H\,\de t-\int_0^T(\ddot u(t),u(t)+\dot u(t))_H\,\de t\nonumber\\
&\quad-\int_0^T(f(t),z(t)+\dot z(t))_H\,\de t+\int_0^T(\ddot u(t),z(t)+\dot z(t))_H\,\de t\nonumber\\
&\quad+\int_0^T(\sigma(t),e z(t)+e\dot z(t))_{p,p'}\,\de t\label{eq:weaksigma}.
\end{align}
We now consider equation~\eqref{unkM} and we use $\varphi=\tau_n(u_n^k+\delta u_n^k-z_n^k-\delta z_n^k)$ as test function. By summing over $k\in\{1,\dots,n\}$ we get
\begin{align*}
\sum_{k=1}^n\tau_n(G^{-1}_n(eu_n^k+e\delta u_n^k),eu_n^k+e\delta u_n^k)_{p,p'}&=\sum_{k=1}^n\tau_n(f_n^k,u_n^k+\delta u_n^k)_H-\sum_{k=1}^n\tau_n(\delta^2u_n^k,u_n^k+\delta u_n^k)_H\\
&\quad-\sum_{k=1}^n\tau_n(f_n^k,z_n^k+\delta z_n^k)_H+\sum_{k=1}^n\tau_n(\delta^2u_n^k,z_n^k+\delta z_n^k)_H\\
&\quad+\sum_{k=1}^n\tau_n(G^{-1}_n(eu_n^k+e\delta u_n^k),ez_n^k+e\delta z_n^k)_{p,p'}.
\end{align*}
By using the notations introduced before, we can rewrite the previous identity as
\begin{align}\label{interp-eq}
\int_0^T(\sigma_n^+(t),G_n(\sigma_n^+(t)))_{p,p'}\,\de t&=\int_0^T(G^{-1}_n(eu_n^+(t)+e\tilde u_n^+(t)),eu_n^+(t)+e\tilde u_n^+(t))_{p,p'}\,\de t\nonumber\\
&=\int_0^T(f_n^+(t),u_n^+(t)+\tilde{u}_n^+(t))_H\,\de t-\int_0^T(\dot{\tilde u}_n(t),u_n^+(t)+\tilde u_n^+(t))_H\,\de t\nonumber\\
&\quad -\int_0^T(f_n^+(t),z_n^+(t)+\tilde z_n^+(t))_H\,\de t+\int_0^T(\dot{\tilde u}_n(t),z_n^+(t)+\tilde z_n^+(t))_H\,\de t\nonumber\\
&\quad +\int_0^T(\sigma_n^+(t),ez_n^+(t)+e\tilde z_n^+(t))_{p,p'}\,\de t.
\end{align}
Now we pass to the limit in~\eqref{interp-eq} as $n\to\infty$. Thanks to the strong convergences
\begin{equation*}
 f_n^+\xrightarrow[n\to\infty]{L^2(0,T;H)}f,\quad z_n^+\xrightarrow[n\to\infty]{L^{p'}(0,T;V)}z,\quad z_n^+\xrightarrow[n\to\infty]{L^2(0,T;H)}z,\quad \tilde z_n^+\xrightarrow[n\to\infty]{L^{p'}(0,T;V)}\dot z,\quad \tilde z_n^+\xrightarrow[n\to\infty]{L^2(0,T;H)}\dot z
\end{equation*}
and the convergences in~\eqref{weak-conv},~\eqref{weak-conv2}, and~\eqref{eq:strong-conv} we deduce that there exists
\begin{align*}
&\lim_{n\to \infty}\int_0^T(\sigma_n^+(t),G_n(\sigma_n^+(t)))_{p,p'}\,\de t\nonumber\\
&=\int_0^T(f(t),u(t)+\dot u(t))_H\,\de t-\int_0^T(\ddot u(t),u(t)+\dot u(t))_H\,\de t\nonumber\\
&\quad -\int_0^T(f(t),z(t)+\dot z(t))_H\,\de t+\int_0^T(\ddot u(t),z(t)+\dot z(t))_H\,\de t+\int_0^T(\sigma(t), ez(t)+e\dot z(t))_{p,p'}\,\de t\\
&=\int_0^T(\sigma(t),e u(t)+e\dot u(t))_{p,p'}\,\de t,
\end{align*}
in view of~\eqref{eq:weaksigma}. Notice that by~\eqref{eq:estM}
\begin{align}
\|G(\sigma_n^+)-eu_n^+-e\tilde u_n^+\|_{L^{p'}(0,T;L^{p'}(\Omega;\R^{d\times d}_{sym}))}^{p'}&=\|G(\sigma_n^+)-G_n(\sigma_n^+)\|_{L^{p'}(0,T;L^{p'}(\Omega;\R^{d\times d}_{sym}))}^{p'}\nonumber\\
&=\frac{1}{n^{p'}}\|\sigma_n^+\|_{L^{p}(0,T;L^{p}(\Omega;\R^{d\times d}_{sym}))}^{p-1}\le \frac{C_1^{p'}}{n^{p'}}\xrightarrow[n\to\infty]{}0\label{ggn},
\end{align}
which gives
\begin{align*}
\lim_{n\to \infty}\int_0^T(\sigma_n^+(t),G(\sigma_n^+(t)))_{p,p'}\de t=\lim_{n\to \infty}\int_0^T(\sigma_n^+(t),G_n(\sigma_n^+(t)))_{p,p'}\,\de t=\int_0^T(\sigma(t),e u(t)+e\dot u(t))_{p,p'}\,\de t.
\end{align*}
Moreover, thanks to (G3) and~\eqref{eq:estM} the sequence $\{G(\sigma_n^+)\}_n\subset L^{p'}(0,T;L^{p'}(\Omega;\R^{d\times d}_{sym}))$ is uniformly bounded. Hence, by~\eqref{weak-conv} and~\eqref{ggn} we derive
$$G(\sigma_n^+)\xrightharpoonup[n\to\infty]{L^{p'}(0,T;L^{p'}(\Omega;\R^{d\times d}_{sym}))}eu+e\dot u,$$
We combine these two facts and we obtain that for all $w\in L^p(0,T;L^p(\Omega;\R^{d\times d}_{sym}))$
\begin{align*}
&0\le \lim_{n\to\infty}\int_0^T(\sigma_n^+(t)-w(t),G(\sigma_n^+(t))-G(w(t)))_{p,p'}\,\de t= \int_0^T(\sigma(t)-w(t),eu(t)+e\dot u(t)-G(w(t)))_{p,p'}\,\de t.
\end{align*}
In particular, we take $w:=\sigma-k b$ with $b\in L^p(0,T;L^p(\Omega;\R^{d\times d}_{sym}))$ and $k>0$, and by dividing by $k$ we get
\begin{align*}
0\le \int_0^T(b(t),eu(t)+e\dot u(t)-G(\sigma(t)-kb(t)))_{p,p'}\,\de t.
\end{align*}
Since $G$ is continuous, by sending $k\to 0^+$ we deduce
\begin{align*}
0\le \int_0^T(b(t),eu(t)+e\dot u(t)-G(\sigma(t)))_{p,p'}\,\de t
\end{align*}
for all $b\in L^p(0,T;L^p(\Omega;\R^{d\times d}_{sym}))$. This implies the constitutive law~\eqref{eq:conlaw}.
\end{proof}

We can finally prove our main existence result Theorem~\ref{thm:main}.

\begin{proof}[Proof of Theorem~\ref{thm:main}]
It is enough to combine Lemma~\ref{lem:convKVF} with Lemmas~\ref{lem:boun-con}--\ref{lem:conlaw}.
\end{proof}

We conclude this section with a uniqueness result in the space $(\mathcal V\cap W^{2,2}(0,T;H))\times L^p(0,T;L^p(\Omega;\R^{d\times d}_{sym}))$ for the weak solutions $(u,\sigma)$ to the system~\eqref{eq:nonlin-KV}--\eqref{eq:boundaryN} which satisfy the initial conditions~\eqref{eq:initial}.

\begin{theorem}\label{thm:uniq}
Let $(u,\sigma)\in(\mathcal V\cap W^{2,2}(0,T;H))\times L^p(0,T;L^p(\Omega;\R^{d\times d}_{sym}))$ be a weak solution to the nonlinear viscoelastic system~\eqref{eq:nonlin-KV}--\eqref{eq:boundaryN} satisfying the initial conditions~\eqref{eq:initial}. Then, the function $u$ is unique. Moreover, if $G$ is strictly monotone, also $\sigma$ is unique.
\end{theorem}

\begin{proof}
 Let $(u_1,\sigma_1),(u_2,\sigma_2)\in(\mathcal V\cap W^{2,2}(0,T;H))\times L^p(0,T;L^p(\Omega;\R^{d\times d}_{sym}))$ be two weak solutions to the nonlinear viscoelastic system~\eqref{eq:nonlin-KV}--\eqref{eq:boundaryN} satisfying the initial conditions~\eqref{eq:initial}. 
 
We fix $s\in(0,T]$. If we set $u:=u_1-u_2\in \mathcal V\cap W^{2,2}(0,T;H)$, by arguing as in~\eqref{eq:strong_form}, we derive that $u$ satisfies the following identity
\begin{equation}\label{eq0}
 \int_0^s(\ddot u(t), \varphi(t))_H\,\de t+\int_0^s(\sigma_1(t)-\sigma_2(t),e\varphi(t))_{p,p'}\,\de t=0 
\end{equation}
for all $\varphi\in L^{p'}(0,s;V)\cap L^2(0,s;H)$ with $\varphi(t)\in V_t^D$ for a.e.\ $t\in[0,s]$. Moreover, we have
\begin{equation}\label{cond1}
 u(0)=\dot u(0)=0,\qquad u(t)+\dot u(t)\in V^D_t\quad\text{for a.e.\ $t\in[0,T]$}. 
\end{equation}
Thanks to~\eqref{cond1} we can use $u+\dot u$ as test function in~\eqref{eq0}, and we get
 \begin{equation}\label{eq1}
 \int_0^s(\ddot u(t),u(t)+\dot u(t))_H\,\de t=-\int_0^s(\sigma_1(t)-\sigma_2(t),eu(t)+e\dot u(t))_{p,p'}\,\de t.
\end{equation}
By taking into account~\eqref{eq:monoton} and~\eqref{eq:conlaw}, by~\eqref{cond1} we have
\begin{equation}
 \int_0^s(\sigma_1(t)-\sigma_2(t),eu(t)+e\dot u(t))_{p,p'}\,\de t=\int_0^s(\sigma_1(t)-\sigma_2(t),G(\sigma_1(t))-G(\sigma_2(t)))_{p,p'}\,\de t\geq 0.
\end{equation}
Moroever, since $u\in W^{2,2}(0,T;H)$, we derive
\begin{align}
\int_0^s(\ddot u(t),u(t)+\dot u(t))_H\,\de t=\frac{1}{2}\|\dot u(s)\|_H^2+(\dot u(s),u(s))_H-\int_0^s\|\dot u(t)\|_H^2\,\de t,
\end{align}
and by Young inequality 
\begin{align}\label{eq:u1-young}
|(\dot u(s),u(s))_H|\le \frac{1}{4}\|\dot u(s)\|_H^2+\|u(s)\|_H^2\le \frac{1}{4}\|\dot u(s)\|_H^2+T\int_0^s\|\dot u(t)\|_H^2\,\de t.
\end{align}
Hence, by~\eqref{eq1}--\eqref{eq:u1-young}, for every $s\in(0,T]$ we obtain
 \begin{equation}\label{eq2}
 \frac{1}{4}\|\dot u(s)\|_H^2-(T+1)\int_0^s\|\dot u(t)\|_H^2\,\de t\le \frac{1}{2}\|\dot u(s)\|_H^2+(\dot u(s),u(s))_H-\int_0^s\|\dot u(t)\|_H^2\,\de t \le 0.
\end{equation}
In particular, since
\begin{equation*}
 \frac{\de }{\de s}\left(\e^{-4(T+1)s}\int_0^s\|\dot u(t)\|_H^2\,\de t\right)=\e^{-4(T+1)s}\left(\|\dot u(s)\|_H^2-4(T+1)\int_0^s\|\dot u(t)\|_H^2\,\de t\right)\quad\text{for a.e.\ $s\in[0,T]$},
\end{equation*}
thanks to~\eqref{eq2} we have that the function $s\mapsto \e^{-4(T+1)s}\int_0^s\|\dot u(t)\|_H^2\,\de t $ is decreasing on $[0,T]$, from which we deduce
\begin{equation*}
 \int_0^s\|\dot u(t)\|_H^2\,\de t=0\quad \text{for all $s\in [0,T]$}.
\end{equation*}
Therefore $\dot u\equiv 0$ on $[0,T]$, which implies $u\equiv c$ for some constant $c\in H$. By~\eqref{cond1} we have $0=u(0)=c$, that is $u_1=u_2$. 

Finally, if $G$ is strictly monotone, by $G(\sigma_1)-G(\sigma_2)=eu+e\dot u=0$, we conclude that $\sigma_1=\sigma_2$.
\end{proof}


\section{Energy-dissipation balance and the viscoelastic paradox}\label{sec:enbalandvp}

In Theorem~\ref{thm:main} we proved the existence of a solution $(u,\sigma)$ to the nonlinear viscoelatic system~\eqref{eq:nonlin-KV}--\eqref{eq:boundaryN}. As observed in Lemma~\ref{lem:conlaw}, the displacement $u$ obtained via the discretisation-in-time scheme is more regular in time, more precisely $u\in W^{2,2}(0,T;H)$. This regularity allows us to prove the following energy-dissipation balance. 

\begin{theorem}\label{thm:enbal}
Every weak solution $(u,\sigma)\in (\V\cap W^{2,2}(0,T;H))\times L^{p}(0,T;L^p(\Omega;\R^{d\times d}_{sym}))$ to the nonlinear viscoelastic system~\eqref{eq:nonlin-KV}--\eqref{eq:boundaryN} satisfies the energy-dissipation balance
\begin{equation}\label{eq:enbal}
\frac{1}{2}\|\dot u(s)\|_H^2+\int_0^s(\sigma(t),e\dot u(t))_{p,p'}\,\de t= \frac{1}{2}\|\dot u(0)\|_2^2+\mathcal W(0,s;u,\sigma)\quad\text{for all $s\in[0,T]$},
\end{equation}
where $\mathcal W(0,s;u,\sigma)$ is the total work of $(u,\sigma)$ on the time interval $[0,s]\subseteq[0,T]$, defined as
\begin{align*}
\mathcal W(0,s;u,\sigma):=&\int_0^s(f(t), \dot u(t)-\dot z(t))_H\,\de t+\int_0^s(\ddot u(t), \dot z(t))_H\,\de t+\int_0^s(\sigma(t),e \dot z(t))_{p,p'}\,\de t\quad\text{for all $s\in[0,T]$}.
\end{align*}
\end{theorem}

\begin{proof}
We fix $s\in (0,T]$. By arguing as in~\eqref{eq:strong_form}, we derive that the pair $(u,\sigma)\in (\V\cap W^{2,2}(0,T;H))\times L^{p}(0,T;L^p(\Omega;\R^{d\times d}_{sym}))$ satisfies
\begin{equation*}
\int_0^s(\ddot u(t),\varphi(t))_H\,\de t+\int_0^s(\sigma(t),e \varphi(t))_{p,p'}\,\de t=\int_0^s(f(t),\varphi (t))_H\,\de t
\end{equation*}
for all $\varphi\in L^{p'}(0,s;V)\cap L^2(0,s;H)$ with $\varphi(t)\in V_t^D$ for a.e.\ $t\in[0,s]$.
Hence, if we use $\varphi:=\dot u-\dot z$ we obtain 
\begin{equation*}
\int_0^s(\ddot u(t),\dot u(t))_H\,\de t+\int_0^s(\sigma(t),e \dot u(t))_{p,p'}\,\de t=\mathcal W(0,s;u,\sigma)\quad\text{for all $s\in[0,T]$}.
\end{equation*}
Finally, since $u\in W^{2,2}(0,T;H)$, we can use the identity
\begin{align*}
\int_0^s(\ddot u(t),\dot u(t))_H\,\de t=\frac{1}{2}\|\dot u(s)\|_H^2-\frac{1}{2}\|\dot u(0)\|_2^2\quad\text{for all $s\in [0,T]$}
\end{align*}
to derive~\eqref{eq:enbal}.
\end{proof}

We conclude the paper by showing that in the nonlinear Kelvin-Voigt model, which is the one associated to the monotone operator
\begin{equation}\label{eq:G-KV}
 G(\xi):=|\xi|^{p-2}\xi \quad\text{for $\xi\in\R^{d\times d}_{sym}$}, 
\end{equation}
the solution to the system~\eqref{eq:nonlin-KV}--\eqref{eq:boundaryN} found in Theorem~\ref{thm:main} satisfies another energy-dissipation balance, which is~\eqref{eq:enbal2}. This implies that the crack can not propagate in time, i.e., also the nonlinear Kelvin-Voigt model of dynamic fracture exhibits the viscoelatic paradox, as discussed in the introduction.

We assume that $G$ is defined by~\eqref{eq:G-KV}. Therefore, $G$ satisfies the assumptions (G1)--(G3) and in addition it is strictly monotone. In particular, $G$ is invertible and its inverse is given by
$$G^{-1}(\eta)=|\eta|^{p'-2}\eta \quad\text{for $\eta\in\R^{d\times d}_{sym}$}.$$
In this case, the system~\eqref{eq:nonlin-KV} reduces to 
\begin{equation}\label{eq:nonlin-KV2}
\ddot u(t)-\div(|eu(t)+e\dot u(t)|^{p'-2}(eu(t)+e\dot u(t)))=f(t)\quad\text{in $\Omega\setminus\Gamma_t$, $t\in[0,T]$},
\end{equation}
with boundary conditions
\begin{alignat}{4}
&u(t)=z(t) && \quad \text{on $\partial_D\Omega$}, & \quad t\in[0,T],\\
&|eu(t)+e\dot u(t)|^{p'-2}(eu(t)+e\dot u(t))\nu=0 &&\quad\text{on $\partial_N\Omega\cup \Gamma_t$}, & \quad t\in[0,T],\label{eq:boundaryN2}
\end{alignat}
and initial conditions
\begin{alignat}{4}
&u(0)=u^0,\quad \dot{u}(0)=u^1&&\quad\text{in $\Omega\setminus \Gamma_0$}.\label{eq:initial2}
\end{alignat}
According to Definition~\ref{def:weak_sol}, we say that $u\in\V$ is a {\it weak solution} to the nonlinear Kelvin-Voigt system~\eqref{eq:nonlin-KV2}--\eqref{eq:boundaryN2} if $u(t)-z(t)\in V_t^D$ for all $t\in[0,T]$ and the following identity holds: 
\begin{equation*}
-\int_0^T(\dot u(t),\dot \varphi(t))_H\,\de t+\int_0^T(|eu(t)+e\dot u(t)|^{p'-2}(eu(t)+e\dot u(t)),e \varphi(t))_{p,p'}\,\de t=\int_0^T( f(t),\varphi (t))_H\,\de t
\end{equation*}
for all $\varphi \in \D$. By Theorems~\ref{thm:main} and~\ref{thm:uniq} we know that there exists a unique weak solution $u\in \mathcal V\cap W^{2,2}(0,T;H)$ to~\eqref{eq:nonlin-KV2}--\eqref{eq:boundaryN2} which satisfies the initial conditions~\eqref{eq:initial2}. Moreover, by Theorem~\ref{thm:enbal} the function $u$ satisfies the .

We want to show that the energy-dissipation balance~\eqref{eq:enbal} can be rephrased just in terms of $u$. Given $u\in \mathcal V\cap W^{2,2}(0,T;H)$, we define the mechanical energy $\mathscr E$ at time $s\in[0,T]$ as
$$\mathscr E(s;u):=\frac{1}{2}\|\dot u(s)\|_H^2+\frac{1}{p'}\|eu(s)\|_{p'}^{p'},$$
the energy dissipated by the viscous term $\mathscr V$ on the time interval $[0,s]\subseteq[0,T]$ as
$$\mathscr V(0,s;u):=\int_0^s(|eu(t)+e\dot u(t)|^{p'-2}(eu(t)+e\dot u(t))-|eu(t)|^{p'-2}eu(t),e\dot u(t))_{p,p'}\,\de t,$$
and the total work $\mathscr W$ on the time interval $[0,s]\subseteq[0,T]$ as
\begin{align*}
\mathscr W(0,s;u):=&\int_0^s(f(t), \dot u(t)-\dot z(t))_H\,\de t+\int_0^s(\ddot u(t), \dot z(t))_H\,\de t \nonumber\\
&+\int_0^s(|eu(t)+e\dot u(t)|^{p'-2}(eu(t)+e\dot u(t)),e \dot z(t))_{p,p'}\,\de t.
\end{align*}

\begin{remark}
For $p=2$ we have
$$\mathscr V(0,s;u)=\int_0^s\|e\dot u(t)\|^2_H\, \de t,$$
which corresponds to the viscous dissipation term in the linear Kelvin-Voigt model. Moreover, since $G^{-1}$ satisfies (G1), we deduce that
\begin{equation*}
(G^{-1}(\eta_1)-G^{-1}(\eta_2))\cdot (\eta_1-\eta_2)\ge 0\quad\text{for all $\eta_1,\eta_2\in\R^{d\times d}_{sym}$},
\end{equation*}
and by choosing $\eta_1=eu(t)+e\dot{u}(t)$ and $\eta_2=eu(t)$ we derive 
$$\mathscr V(0,s;u)\ge 0\quad\text{for every $s\in[0,T]$}.$$
Therefore $\mathcal V$ can be seen as the analogous of the viscous dissipation term in the nonlinear setting.
\end{remark}

Thanks to Theorem~\ref{thm:enbal} and~\eqref{eq:G-KV}, we derive the following result.

\begin{corollary}\label{coro:enbal2}
Every weak solution $u\in \V\cap W^{2,2}(0,T;H)$ to the nonlinear Kelvin-Voigt system~\eqref{eq:nonlin-KV2}--\eqref{eq:boundaryN2} satisfies the energy-dissipation balance
\begin{equation}\label{eq:enbal2}
\mathscr E(s;u)+\mathscr V(0,s;u)= \mathscr E(0;u)+\mathscr W(0,s;u)\quad\text{for all $s\in[0,T]$}.
\end{equation}
\end{corollary}

\begin{proof}
By Theorem~\ref{thm:enbal} we know that $u$ satisfies the energy dissipation balance~\eqref{eq:enbal}. Moreover, for the nonlinear operator $G$ given by~\eqref{eq:G-KV} we observe that
\begin{align*}
&\int_0^s(\sigma(t),e\dot u(t))_{p,p'}\,\de t\\
&=\int_0^s(|eu(t)+e\dot u(t)|^{p'-2}(eu(t)+e\dot u(t)),e\dot u(t))_{p,p'}\,\de t\\
&=\int_0^s(|eu(t)+e\dot u(t)|^{p'-2}(eu(t)+e\dot u(t))-|eu(t)|^{p'-2}eu(t),e\dot u(t))_{p,p'}\,\de t\\
&\quad+\int_0^s(|eu(t)|^{p'-2}eu(t),e\dot u(t))_{p,p'}\,\de t\\
&=\int_0^s(|eu(t)+e\dot u(t)|^{p'-2}(eu(t)+e\dot u(t))-|eu(t)|^{p'-2}eu(t),e\dot u(t))_{p,p'}\,\de t+\frac{1}{p'}\|eu(s)\|_{p'}^{p'}-\frac{1}{p'}\|eu(0)\|_{p'}^{p'}.
\end{align*}
Indeed, $u\in W^{1,p'}(0,T;V)$, which implies that the map $t\mapsto\|eu(t)\|_{p'}^{p'}$ is absolutely continuous on $[0,T]$ with 
$$\frac{\de}{\de t}\|e u(t)\|_{p'}^{p'}=p'(|eu(t)|^{p'-2}eu(t),e\dot u(t))_{p,p'}\quad\text{for a.e.\ $t\in[0,T]$}.$$
By combining the previous identity with~\eqref{eq:enbal} we derive~\eqref{eq:enbal2}.
\end{proof}

As a consequence of Corollary~\ref{coro:enbal2} we deduce that for every weak solution $u\in \mathcal V\cap W^{2,2}(0,T;V)$ to the nonlinear Kelvin-Voigt system~\eqref{eq:nonlin-KV2}--\eqref{eq:boundaryN2} the crack can not grow in time. Indeed, as explained in the introduction, according to the Griffith criterion there is a balance between the mechanical energy dissipated and the energy used to increase the crack. In the nonlinear Kelvin-Voigt model~\eqref{eq:nonlin-KV2}--\eqref{eq:boundaryN2}, this reads as
$$\mathscr E(s;u)+\mathcal H^{d-1}(\Gamma_t\setminus\Gamma_0)+\mathscr V(0,s;u)= \mathscr E(0;u)+\mathscr W(0,s;u)\quad\text{for all $s\in[0,T]$}.$$
Since the energy dissipated by the crack growth, which is $\mathcal H^{d-1}(\Gamma_t\setminus\Gamma_0)$, does not appear in~\eqref{eq:enbal2}, we derive that for the weak solution $u\in \mathcal V\cap W^{2,2}(0,T;H)$ to~\eqref{eq:nonlin-KV2}--\eqref{eq:boundaryN2} given by Theorem~\ref{thm:main} we must have $\mathcal H^{d-1}(\Gamma_t\setminus\Gamma_0)=0$ for every $t\in[0,T]$. Hence, the crack associated to $u$ does not increase in time. We point out that this phenomenon, called viscoelastic paradox, is the same which arises in linear Kelvin-Voigt models, as shown in~\cite{DM-Lar,Tasso}.


\begin{acknowledgements}
The authors are members of the {\em Gruppo Nazionale per l'Analisi Ma\-te\-ma\-ti\-ca, la Probabilit\`a e le loro Applicazioni} (GNAMPA) of the {\em Istituto Nazionale di Alta Matematica} (INdAM). M.~C. acknowledges the support of the project STAR PLUS 2020 - Linea 1 (21-UNINA-EPIG-172) ``New perspectives in the Variational modeling of Continuum Mechanics'', and of the INdAM - GNAMPA Project ``Equazioni differenziali alle derivate parziali di tipo misto o dipendenti da campi di vettori'' (Project number CUP\_E53C22001930001). A.C. acknowledges the support of the INdAM - GNAMPA Project ``Problemi variazionali per funzionali e operatori non-locali'' (Project number CUP\_E53C22001930001) and of the MUR PRIN project ``Elliptic and parabolic problems, heat kernel estimates and spectral theory'' (Project number 20223L2NWK). F.S acknowledges the financial support received from the Austrian Science Fund (FWF) through the project TAI 293. 
\end{acknowledgements}


\end{document}